\documentclass[]{article}
\usepackage{graphicx}
\usepackage{caption}
\usepackage{subcaption}
\usepackage{amsfonts}
\usepackage{amsmath}
\usepackage{amssymb}
\usepackage{fancyhdr}
\usepackage{titlesec}
\usepackage{indentfirst}
\usepackage{booktabs}
\usepackage{verbatim}
\usepackage{color}
\usepackage{tcolorbox}
\usepackage{mathtools}
\usepackage{bm}
\usepackage{amsthm}
\usepackage{wasysym}
\usepackage{empheq}
\graphicspath{{figures/}} 
\usepackage[font=small,labelfont=bf]{caption} 
\usepackage{chngcntr}

\newcommand{\rom}[1]{\uppercase\expandafter{\romannumeral #1\relax}}
\usepackage[page,header]{appendix}

\usepackage{titletoc}

\numberwithin{equation}{section}
\newtheorem{theorem}{Theorem}[section]
\newtheorem{lemma}[theorem]{Lemma}

\newtheorem{proposition}[theorem]{Proposition}

\newtheorem{remark}[theorem]{Remark}

\def\a{\alpha}
\def\b{\beta}
\def\d{\delta}
\def\e{\epsilon}
\def\g{\gamma}
\def\l{\lambda}

\def\p{\phi}
\def\r{\rho}
\def\s{\sigma}

\def\bx{{\bf{x}}}
\def\bv{{\bf{v}}}
\def\bz{{\bf{z}}}

\def\bl{\bf{l}}
\def\bi{\bf{i}}

\def\bk{\bf{k}}

\def\R{\mathbb R}

\def\mL{\mathcal{L}}

\def\mF{\mathcal{F}}

\def\la{\left\langle}
\def\ra{\right\rangle}
\def\ll{\left\lVert}
\def\rl{\right\rVert}
\def\lv{\left\lvert}
\def\rv{\right\rvert}
\def\({\left(}
\def\){\right)}
\def\[{\left[}
\def\]{\right]}

\def\pt{\partial}

\def\cd{\cdot}

\def\qd{\quad}

\def\t{\tilde}

\def\sM{\sqrt{M}}

\def\bu{u}

\def\b{\epsilon}

\begin{document}

\title{Hypocoercivity and Uniform Regularity for the Vlasov-Poisson-Fokker-Planck System with Uncertainty and Multiple Scales \footnote{This work was partially supported by NSF grants DMS-1522184 and DMS-1107291: RNMS KI-Net, by NSFC grant No. 91330203, and by the Office of the Vice Chancellor for Research and Graduate Education at the University of Wisconsin-Madison with funding from the Wisconsin Alumni Research Foundation.}}

\author{Shi Jin\footnote{Department of Mathematics, University of Wisconsin-Madison, Madison, WI 53706, USA (sjin@wisc.edu), and  Institute of Natural Sciences, Department of Mathematics, MOE-LSEC and SHL-MAC, Shanghai Jiao Tong University, Shanghai 200240, China}  \ \ and \     Yuhua Zhu\footnote{Department of Mathematics, University of Wisconsin-Madison, Madison, WI 53706, USA (yzhu232@wisc.edu)}}
\date{\today}
\maketitle

\begin{abstract}
We study the Vlasov-Poisson-Fokker-Planck system with uncertainty and multiple scales. Here the uncertainty, modeled by random variables, enters the solution through initial data, while the multiple scales lead the system to its
high-field or parabolic regimes. With the help of  proper Lyapunov-type inequalities, under some mild conditions on the initial data, the regularity of the solution in the random
space, as well as exponential decay of the solution to the global Maxwellian,
are established under Sobolev norms,
 which are {\it uniform} in terms of the scaling parameters. These are the
first hypocoercivity results for a nonlinear kinetic system with random
input, which are important for the understanding of
the sensitivity of
the system under random perturbations, and for the establishment of spectral
convergence of popular numerical methods for uncertainty quantification based
on (spectrally accurate) polynomial chaos expansions.

\end{abstract}

{\small

{\bf AMS subject classifications.}

}

\section{Introduction}
\label{sec:Intro}

In this paper we are interested in the Vlasov-Poisson-Fokker-Planck (VPFP) system with random inputs. The VPFP system describes the Brownian motion of a large system of particles in a surrounding bath. One of the applications is in electrostatic plasma, in which one considers the interactions between the electrons and a surrounding bath via the Coulomb force \cite{Chand}. The uncertainty in
a kinetic equation  can arise from the initial and boundary data, the forcing term, collisional kernels, etc,  due to modeling and measurement errors. In this paper we will mainly focus on the case in which the initial data contain random inputs, modeled by random variables with given probability density functions.
The goal is to understand the regularity of the solution in the random space,
as well as its long-time behavior.  Such a study is important in order to understand the
{\it sensitivity} of the system under random perturbations. It is also the
basis to study the convergence of numerical schemes for such problems,
for example, the popular methods for uncertainty quantification, such as
polynomial chaos expansion based stochastic Galerkin or stochastic collocation
methods \cite{GS, GWZ, xiu2002wiener, Xiu}, which enjoy a spectral convergence,  if the
solution has the desired regularity in the random space.

While there have been many developments in the regularity of the solution to elliptic or parabolic equations with uncertainties \cite{babuska2004galerkin, cohen2010convergence, cohen2011analytic}, such study has been scarce for hyperbolic type equations \cite{GotXiu, TangZhou, ZhouTang, branicki2013fundamental, despres2015uncertainty} because of the poor regularity of the solution. The uncertainty quantification, while
popular in many types of partial differential equations, has seldomly been
studied for kinetic equations until very recently \cite{zhu2016vlasov, hu2016stochastic, jin2015asymptotic, jin2016asymptotic}. Typically kinetic
equations possess multiple scales, leading to various different asymptotic regimes,
demanding carefully designed numerical methods to handle different asymptotic
behavior of the equations. For deterministic kinetic equations, one efficient
multiscale paradigm is the {\it Asymptotic-Preserving} schemes, which
mimic the asymptotic transitions from kinetic equations  to their  diffusion or hydrodynamic limits in the numerically discrete space \cite{Jin99, jin2010asymptotic}. This concept was extended
to random kinetic equations in \cite{ jin2015asymptotic}, in the framework of
{\it stochastic Asymptotic-Preserving} methods. Convergence study of
these methods clearly requires the understanding of the regularity of
the solution. Moreover, the correct asymptotic behavior of the numerical
methods in various asymptotic regimes also require the understanding of
the long time behavior, and how the decay rates depend on the small scaling parameters. For linear transport equation with random isotropic scattering in diffusive regime, such regularity and asymptotic
behavior were first studied in \cite{jin2016august}, in which the regularity
of the solution was established, as well as its exponential decay    toward the
local equilibrium, all {\it uniformly} in the mean free path (or Knudsen
number). Uniform regularity for the semiconductor Boltzmann equation, in which
the scattering is anisotropic and random,  was established
in \cite{Jin-LiuL}.  Called
{\it hypocoercivity} by Villani \cite{dric2009hypocoercivity}, the property
of uniform
exponential decay toward the global equilibrium \cite{dolbeault2015hypocoercivity} was further explored in \cite{li2017uniform} for general linear
kinetic equations with uncertainty. So far there has been no work on
hypocoercivity for nonlinear kinetic equations with uncertainty
 with uniform (in small scaling parameters)
estimate. The purpose
of this paper is to conduct such a study for the nonlinear VPFP  system with random
initial input. 

Depending on different scales, the VPFP system possesses two distinguished
asymptotic limits, the high field limit and the parabolic limit. We will
treat these different scalings in a unified framework.
With the help of  proper Lyapunov-type inequalities, we first develop two energy estimates for the microscopic (VPFP) and macroscopic (limiting) systems, which allows us to obtain the {\it uniform}--in
terms of the scaling parameters--regularity in the random space
 of the perturbative solution of the nonlinear VPFP system near the global
Maxwellian. Under some mild conditions on the initial data, we found that the solution will decay exponentially to the global Maxwellian in a rate
{\it independent} of the small scaling parameter. Our results also reveal
that the initial random perturbation will die out exponentially in time,
uniformly in the scaling parameter, thus
the solution is insensitive to the initial random perturbation, in all
asymptotic regimes.

For the deterministic VPFP system, the regularity and convergence toward the global
Maxwellian or asymptotic limits  were conducted in, for examples, \cite{arnold2001low, duan2010kinetic, goudon2005multidimensional,
nieto2001high, soler1997asymptotic, hwang2013vlasov, renjun2015time}.
Our energy estimates rely on the hypocoercivity results  of \cite{duan2010kinetic}, and the energy estimates in \cite{hwang2013vlasov} with suitable modification to effectively separate the microscopic and macroscopic scales in order
to get better estimates in the asymptotic regimes.
When the small scaling parameters are involved, which was not considered
in \cite{hwang2013vlasov}, it is crucial to 
get rid of the bad dependence on these parameters in the initial condition and rate of convergence to the global
equilibrium. Therefore  we  have not only  extended the regularity results to the random space, but also improved the micro-macro energy estimates by separating the  microscopic energy from the macroscopic energy suitably, so
when the small scales are involved, we can
get uniform convergence rate towards the global equilibrium,
and a milder initial condition at the same time. As a result,
we get an exponential decay of the perturbative solution--independent of the small parameter-- under some mild initial condition, which leads to a uniform regularity of the solution in random space for both high field and parabolic limits.

In this paper, for clarity of the presentation and notations, we carry out all analysis in one space dimension for all independent variables.  Its extension to higher dimension in $x, v$ and $z$ is straightforward with some changes of the constants (see \cite{Zhu_BE_UQ} for example). 

This paper is organized as follows. Section \ref{models} gives an introduction of the VPFP system with uncertainty and its two different asymptotic regimes. The main results are stated in Section \ref{main results}. Then in Sections \ref{sec:micro} - \ref{sec:macro} we prove the energy estimates from microscopic and macroscopic systems respectively. The uniform regularity of the perturbative solution is obtained in Section \ref{sec: expo decay}.


\section{The VPFP System with Uncertainty and Asymptotic Scalings}
\label{models}
\subsection{The VPFP System with Uncertainty}
In the dimensionless VPFP system with uncertainty, the time evolution of particle density distribution function $f(t,x,v,z)$ under the action of an electrical potential $\p(t,x, z)$  satisfies
\begin{equation}
\begin{cases}
&\partial_tf + \frac{1}{\d} v\pt_xf - \frac{1}{\b} \pt_x\p\pt_vf = \frac{1}{\d\b}\mF f, \label{eq:VPFP}\\
&-\pt_x^2\p= \rho - 1, \quad t>0, \quad \bx, \bv \in  \mathbb{R}, \ \bz \in I_\bz\subseteq\mathbb{R},
\end{cases}
\end{equation}
with initial data:\\
\begin{equation}
f(0, x, v, z) = f^0( x, v, z), \quad x, v\in \mathbb{R}, \ z \in I_\bz\subseteq\mathbb{R} .
 \label{eq:VPFP_bc}
\end{equation}

The distribution function $f(t x,v, z)$ depends on time $t$, position $x$, velocity $v$ and random variable $z$. $\p(t, x, z)$ is a self-consistent electrical potential and  $\rho(t, x, z)$ is  the density function defined as
 \begin{align}
 \rho(t,x, z) = \int_{\mathbb{R}^N} f (t, x, v, z) d v.
 \label{rho_def}
 \end{align}

 In the VPFP system, $\mF$ is a collision operator, describing the Brownian motion of the particles, which reads,
 \begin{align}
 \mF f =\pt_v\left( M \pt_v \left( \frac{f}{M}\right)\right),
 \label{define L}
 \end{align}
 where $M$ is the {\it global equilibrium} or {\it global Maxwellian},
 \begin{align}
M = \frac{1}{\sqrt{2\pi}}e^{-\frac{|v|^2}{2}}.
\end{align}

In the dimensionless system, $\d$ is the reciprocal of the scaled thermal velocity, $\b$ represents the scaled thermal mean free path \cite{soler1997asymptotic}. There are two different regimes for this system. One is the {\it{high field regime}}, where $\d =1$.
As $\b$ goes to zero, $f$ goes to the local Maxwellian $M_{\text{local}} =\frac{\rho}{\sqrt{2\pi}}e^{-\frac{|v - \pt_x \p|^2}{2}}$, and the VPFP system converges to a hyperbolic limit \cite{arnold2001low, goudon2005multidimensional,nieto2001high}:
\begin{equation}
\begin{cases}
&\pt_t \r + \pt_x\left(  \r \pt_x\p \right) = 0, \\
&-\pt^2_x\p= \rho - 1,
\end{cases}
\end{equation}
Another regime is the {\it{parabolic regime}}, where
$\d = \b$.
When $\b$ goes to zero,  $f$ goes to the global Maxwellian $M$, and the VPFP system converges to a parabolic limit \cite{PouSol}:
\begin{equation}
\begin{cases}
&\pt_t \r - \pt_x\(\pt_x\r - \r\pt_x\p \right) = 0, \\
&-\pt_x^2\p= \rho - 1.
\end{cases}
\end{equation}
In this paper, we are going to study both regimes together.

In the VPFP system with uncertainty,  the random variable $\bf{z}$ is in a properly defined probability space $(\Sigma, \mathbb{A}, \mathbb{P})$, whose event space is $\Sigma$ and is equipped with $\sigma$-algebra $\mathcal{A}$ and probability measure $\mathbb{P}$. Define $\pi(\bz): I_{\bf{z}} \longrightarrow \mathbb{R}^+$ as the probability density function of the random variable $\bf{z}(\omega)$, $\omega\in\Sigma$. So one has a corresponding $L^2$ space in the measure of,
 \begin{align}
d\mu =  d \mu(x, v, z) = \pi(z) d x \,d v\, dz .
\label{measure}
 \end{align}
With this measure, one has the corresponding Hilbert space with the following inner product and norms:
\begin{align}
\la f, g \ra  = \int_{\R}\int_{\R}\int_{I_\bz} fg \, d\mu(x,v,z),  \quad \text{or}, \quad \la \r, j \ra  =\int_{\R}\int_{I_\bz} \r j \, d\mu(x, z), \quad \text{with norm }\ll f \rl^2 = \la f , f \ra.
\label{regular norm}
\end{align}

For convenience of the readers, we list some elementary calculation on $M$ which will be used in later calculations:
\begin{align}
&\pt_v{M} = -vM,\quad \pt_v(\sM) = -\frac{v}{2}\sM;\quad\\
&\int_{\R} v^a\sM \, dv = \int_\R v^aM \, dv = 0, \quad \text{ for any odd }a;\\
& \int_{\R} M \, dv = 1, \quad \int_{\R} v^2M \, dv = 1,\quad \int_{\R}v^4Mdv = 3;\\
&\int_{\R} |v|^3M \, dv = \frac{4}{\sqrt{2\pi}}\leq 2,\quad \int_{\R} \left(\pt_v(v\sM)\right)^2\, dv = \frac{3}{4}.
\end{align}

\subsection{Notations}
In this paper, we only focus on one space dimension. Without loss of generality, we assume $\b<1$. 
In order to get the convergence rate of the solution to the global equilibrium, we define,
\begin{align}
h = \frac{f - M}{\sqrt{M}}, \quad \s = \int_{\R} h\sM \, dv, \quad \bu = \int_{\R} h\, v\sM \, dv,
\label{def u}
\end{align}
 where $h$ is the fluctuation around the equilibrium, $\s$ is the density fluctuation, $\bu$ is the velocity fluctuation.
Then the microscopic quantity $h$ satisfies,
\begin{empheq}[left=\empheqlbrace]{align}
&\b\d\underbrace{\pt_t h}_{\rom{1}} + \b \underbrace{v\pt_xh}_{\rom{2}} -\d\underbrace{\pt_x\p\pt_vh}_{\rom{3}} +  \d\underbrace{\frac{v}{2}\pt_x\p h}_{\rom{4}} +  \d\underbrace{v\sM \pt_x\p}_{\rom{5}} = \underbrace{\mL h}_{\rom{6}},\label{with e on t_1}\\
&\pt_x^2\p = -\s,
\label{with e on t_2}
\end{empheq}
where $\mL$ is the so-called linearized Fokker-Planck operator,
\begin{align}
\mL h = \frac{1}{\sM}\mF \left(M + \sM h \right) = \frac{1}{\sM}\pt_v \( M\pt_v \left( \frac{h}{\sM} \right) \).
\end{align}
We give each term a number, in order to make it clear where the term comes from originally when doing the energy estimates later.

 We further introduce projections onto $\sM$ and $v\sM$,
\begin{align}
\Pi_1 h = \s \sM, \quad \Pi_2 h = \bu v\sM, \quad \Pi h = \Pi_1 h + \Pi_2 h.
\end{align}
These projections have the following properties:
\begin{itemize}
\item [--] $\pt_x\pt_z\Pi= \Pi\pt_x\pt_z$
\item [--] Due to the mutual orthogonality of $\Pi_1 h$, $\Pi_2h$, $(1-\Pi)h$ in $L_v^2$ space,  let $\pt^{\bk} = \pt_z^{k_1}\pt_x^{k_2}$,
\begin{align}
\ll \pt^{\bk}h \rl_{L^2_v}^2 =&  \ll \Pi_1 \pt^{\bk}h\rl_{L^2_v}^2 + \ll \Pi_2\pt^{\bk}h \rl_{L^2_v}^2 + \ll (1-\Pi)\pt^{\bk}h \rl_{L^2_v}^2 \nonumber\\
=& \ll \pt^{\bk}\s\rl_{L^2_v}^2 + \ll \pt^{\bk}\bu \rl_{L^2_v}^2 + \ll (1-\Pi)\pt^{\bk}h \rl_{L^2_v}^2,
\end{align}
which also implies,
\begin{align}
\ll \pt^{\bk}\s \rl_{L^2_v}, \ll \pt^{\bk}\bu \rl_{L^2_v}, \ll (1-\Pi)\pt^{\bk}h \rl_{L^2_v} \leq \ll \pt^{\bk}h \rl_{L^2_v}.
\label{u s h}
\end{align}
\end{itemize}
Multiplying $\sM$ and $v\sM$ to (\ref{with e on t_1}), and integrating the equation over $v$ respectively, then one has the equations for the macroscopic quantities $\s$ and $\bu$,
\begin{empheq}[left=\empheqlbrace]{align}
&\d\pt_{t} \s +  \pt_x\bu = 0,
 \label{eqn for u with e_1}\\
 &\b\d\underbrace{\pt_{t} \bu}_{\rom{1}} +  \b \underbrace{\pt_x\s}_{\rom{2}.1} + \b\underbrace{\int v^2\sM (1-\Pi)\pt_xh dv}_{\rom{2}.2}  + \d \underbrace{ \pt_x\p\s}_{\rom{3}} + \underbrace{\bu}_{\rom{6}} + \d\underbrace{\pt_x\p}_{\rom{5}}= 0.
 \label{eqn for u with e_2}
\end{empheq}
We call (\ref{with e on t_1})-(\ref{with e on t_2}) the {\it microscopic system}, and (\ref{eqn for u with e_1})-(\ref{eqn for u with e_2}) the {\it macroscopic system}.  Note (\ref{eqn for u with e_1})-(\ref{eqn for u with e_2}) are not a closed system since it contains the microscopic quantities $h$.

We also define the following norms and energies,
\begin{itemize}
\item Norms:
\begin{itemize}
\item  $\lv h \rv^2 = \int_{\R} h^2 \,dv$, \qd $\ll f \rl$ and $\la \cd, \cd\ra$ is defined in (\ref{regular norm}). 
\item $\lv h \rv_\nu^2 = \int_\R (1+|v|^2)h^2 + \(\pt_vh\)^2 \, dv$, \qd $\ll h\rl_\nu^2 = \int_{\R\times I_z} \lv h \rv_\nu d\mu(x,z),$
\item $\ll f\rl_{H_z^m}^2 = \sum_{l=0}^m\ll \pt_z^lf \rl^2$, \quad $\ll f \rl^2_{H^m_z(H^n_x)} = \sum_{i\leq n}\ll \pt^i_x f \rl^2_{H^m_z} $,
\end{itemize}
\item Energy terms:
\begin{itemize}
\item $E^{m,i}_h= \ll \pt_x^ih \rl^2_{H_z^m}, \qd E^m_h = \ll h \rl_{H^m_z(H^1_x)} = E^{m,0}_h + E^{m,1}_h,$
\item $E^{m,i}_\p= \ll \pt_x^i\pt_x\p \rl^2_{H_z^m}$, \qd $E^m_\p= \ll \pt_x\p \rl_{H^m_z(H^1_x)} = E^{m,0}_\p + E^{m,1}_\p $;
\end{itemize}
\item Dissipation terms:
\begin{itemize}
\item $D^{m,i}_h = \sum_{l \leq m}\ll\pt_z^l \pt_x^i(1-\Pi)h \rl_\nu^2$, \qd$D^m_h = D^{m,0}_h + D^{m,1}_h,$
\item $D^{m,i}_\p= \ll \pt_x^i\pt_x\p \rl^2_{H_z^m}$, \qd $D^m_\p = D^{m,0}_\p + D^{m,1}_\p $;
\item $D^{m,i}_u = \ll \pt_x^iu \rl_{H^m_z}$, \qd $D^m_u = D^{m,0}_u + D^{m,1}_u$;
\item $D^{m,i}_\s = \ll \pt_x^i\s \rl_{H^m_z}$, \qd $D^m_\s = D^{m,0}_\s + D^{m,1}_\s$;
\end{itemize}
\end{itemize}

\section{Main Results}
\label{main results}
To get the regularity of the solution in the Hilbert space, one usually uses energy estimates. In order to balance the nonlinear term $\pt_x\p \pt_vf$, and get a regularity independent of the small parameter $\b$ (or depending on $\b$ in a good way), one needs the coercivity property from the collision operator. The coercivity property one uses most commonly is
\begin{align}
{-\int_\R  h \mL h\, dv \geq C \lv (1-\Pi_1)h\rv^2,}
\end{align}
see \cite{dolbeault2015hypocoercivity, dric2009hypocoercivity}. However, this is not enough for the non-linear case. We need  stronger coercivity as listed in the following Proposition, see  \cite{duan2010kinetic} for deterministic case. {Here we extend the coercivity into random space.} 
\begin{proposition}
\label{property of L}
For $\mL$ defined in (\ref{define L}),
\begin{itemize}
\item[(a)] $-\la \mL h, h \ra = -\la \mL(1-\Pi)h , (1-\Pi)h \ra + \ll \bu \rl^2$;
\item[(b)] $-\la \mL(1-\Pi)h , (1-\Pi)h \ra  = \ll \pt_v(1-\Pi)h \rl^2 + \frac{1}{4}\ll v(1-\Pi)h\rl^2 - \frac{1}{2}\ll (1-\Pi)h \rl^2$;
\item[(c)] $-\la \mL(1-\Pi)h , (1-\Pi)h \ra \geq \ll (1-\Pi)h \rl^2$;
\item[(d)] There exists a constant $\l_0>0$, such that the following hypocoercivity holds,
\begin{align}
-\la \mL h , h\ra \geq & \l_0\ll (1-\Pi)h \rl^2_\nu +\ll \bu \rl^2,
\label{hypo}
\end{align}
and  the  largest $\l_0 = \frac{1}{7}$ in one dimension.
\end{itemize}
\end{proposition}
\begin{proof}
Here we only prove $(d)$, see \cite{duan2010kinetic} for $(a), (b), (c)$.
Since
\begin{align}
&-\la \mL(1-\Pi)h, (1-\Pi)h \ra \nonumber\\
\geq &{-a\la \mL(1-\Pi)h, (1-\Pi)h \ra - (1-a)\la \mL(1-\Pi)h, (1-\Pi)h \ra}\nonumber\\
\geq& a \ll \pt_v(1-\Pi)h \rl^2 + \frac{a}{4}\ll v(1-\Pi)h\rl^2 - \frac{a}{2}\ll (1-\Pi)h \rl^2 + (1-a)\ll (1-\Pi)h \rl^2\nonumber\\
\geq& \min_{0<a<1}\{a, \frac{a}{4}, (1 - \frac{3}{2}a) \} \ll (1-\Pi)h \rl^2_\nu,
\end{align}
{for $a$ to be determined later, where the second inequality is according to (b) and (c).} Then the largest $\l_0$ one can get is when $a = \frac{4}{7}$, $\l_0 = \frac{1}{7}$. Therefore,
\begin{align}
- \la \mL h, h\ra \geq&  \l_0\ll (1-\Pi)h \rl_\nu^2 + \ll \bu \rl^2.
\end{align}

\end{proof}
{Before we go into technique details, we first summarize the main strategy of this paper here, which is mainly based on \cite{hwang2013vlasov}. We omit $\e$, $\d$ to see the main structure of energy estimates first. We want to use energy estimates to analyze the energy $E^m = E^m_h+ E^m_\p$, the goal here is to obtain a Lyapunov-type inequality like,
\begin{align}
    \pt_t{E^m} + D^m \leq \sqrt{E^m}D^m,
\end{align}
so that one can control the initial data to get an uniform regularity. However, if one only does energy estimates for (\ref{with e on t_1}) - (\ref{with e on t_2}), the dissipation from the linearized Fokker-Planck operator, $D^m_h+D^m_u$, cannot bound the nonlinear term $\sqrt{E^m_h + E^m_\p}(D^m_h + D^m_u  + D^m_\s + D^m_\p)$. So we involve the microscopic system (\ref{eqn for u with e_1}) - (\ref{eqn for u with e_2}), where the dissipation terms $D^m_\s + D^m_\p$ comes from $\rom{2}.1$ and $\rom{5}$ in (\ref{eqn for u with e_1}). Combine the microscopic and macroscopic energy estimates, one ends up with a new energy estimates,
\begin{align}
    \pt_t \hat{E}^m + D^m \leq \sqrt{\hat{E}^m}D^m,
\end{align}
where $\hat{E}^m \sim E^m$, and $D^m = D^m_h + D^m_u + D^m_\s + D^m_\p$, so the non-linearity can be fully controlled by the dissipation terms, which gives what we want. }

{With $\e$ and $\d$ involved, one needs to bound the nonlinear term more carefully, see Lemma \ref{ineqs}, which is the key difference from \cite{hwang2013vlasov}. See Remark \ref{explain diff} for the importance of these careful estimates for the nonlinear term.
}

Based on the coercivity (\ref{hypo}), we have the following two estimates for the microscopic  and macroscopic systems respectively.
\begin{lemma}
\label{lemma:est}
The solution to system (\ref{with e on t_1}) - (\ref{with e on t_2})  satisfies the following estimates, {for any $m\geq 1$,}
\begin{align}
&\frac{1}{2}\pt_{t} \left[  E^m_h+ \frac{\d}{\e} E^m_\p\right] + \frac{\l_0}{\d\e}D^m_h +\frac{1}{\d\e}D^m_\bu\nonumber\\
\leq& \frac{AC_1^2}{a\e}\sqrt{E^m_h} \left(3D^m_\bu +2D^m_h\right) + \frac{2AC_1^2}{\e}\sqrt{E^m_\p}  \left(\(4+\frac{1}{a}\)D^m_\bu +4D^m_h\right)\nonumber\\
&+\frac{aAC_1^2}{\e}\sqrt{E^m_h}D^m_\p + \frac{aAC_1^2}{\e}\sqrt{E^m_\p}D^m_\s.
\label{eqn: micro}
\end{align}
and
\begin{align}
 &\pt_{t}\left[ \sum_{l=0}^{m-1}\la  \pt_z^{l}\bu, \pt_z^{l}\pt_x\p \ra + \sum_{l=0}^m\la \pt_z^l \pt_x\bu, \pt_z^l\pt_x^2\p \ra + \frac{1}{2\e} E^m_\p \right] + \frac{ 1}{2\d}D^m_\s + \frac{1}{\e} D^m_\p  \nonumber\\
  \leq & \frac{1}{\d} D^m_\bu + \frac{1}{2\d}D^m_h  +\frac{AC_1^2}{\e}\sqrt{E^m_\p} \( 3D^m_\s + 2D^m_\p\),
    \label{eqn: macro}
  \end{align}
where
\begin{align}
{A = 2\sqrt{m+1}\binom{m}{[\frac{m}{2}]} + \sqrt{m+1}}
\label{def of A}
\end{align}
is a constant only depending on $m$ and $[m/2] $ is the smallest integer larger or equal to $\frac{m}{2}$, and $C_1$ is the Sobolev constant in one dimension defined in (\ref{def:Ck}).
\end{lemma}

If one combines the above two inequalities, the "bad terms"  on the right hand side (RHS) can be controlled by the dissipation terms on the left hand side (LHS) if the coefficients are carefully balanced. Hence, one can come to the conclusion that the solution exponentially decays to the global equilibrium.


\begin{remark}
The main difference between the energy estimates in Lemma
\ref{lemma:est}
 and the one obtained in \cite{hwang2013vlasov} is that for both micro and macro systems, we separate the microscopic energy $E^m_h$ from the macroscopic energy $E^m_\p$ for $D^m_\p$ and $D^m_\s$, which gives us more flexibility to bound the energies, especially when small parameters are involved.
\end{remark}

\begin{theorem}
For the high field regime ($\d = 1$), if
\begin{align}
E^m_h(0) + \frac{1}{\b^2}E^m_\p(0) \leq   \frac{2\l_0^3}{\(80AC_1\)^2},
\label{initial cond_1}
\end{align}
then,
\begin{align}
E^m_h(t) \leq \frac{3}{\l_0}e^{-\frac{t}{\e}}\( E^m_h(0) + \frac{1}{\e^2}E^m_\p(0)\), \qd E^m_\p(t) \leq \frac{3}{\l_0}e^{-t}\(\e^2 E^m_h(0) + E^m_\p(0)\)
\label{conclusion: HFL}
\end{align}
For the parabolic regime ($\d = \b$), if
\begin{align}
E^m_h(0) + \frac{1}{\b}E^m_\p(0) \leq\(\frac{2\l_0^3}{(80AC_1)^2}\)\frac{1}{\e},
\label{initial cond_2}
\end{align}
then,
\begin{align}
&E^m_h(t) \leq \frac{3}{\l_0}e^{-\frac{t}{\e}}\( E^m_h(0) + \frac{1}{\e}E^m_\p(0)\), \qd E^m_\p(t) \leq   \frac{3}{\l_0}e^{-t}\(\e E^m_h(0) + E^m_\p(0)\).
\label{conclusion: PR}
\end{align}
Here $A$ and $C_1$ are the same as in Lemma \ref{lemma:est}.
\label{expo decay}
\end{theorem}

\begin{remark}
Basically, Theorem \ref{expo decay} implies the following,
\begin{itemize}
\item[(a)] For the High Field regime, as long as initially the electric field $\pt_x\p$ is $O(\e)$ small, and the initial data $f$ is suitably bounded by (\ref{initial cond_1}), then the solution will converge to the global equilibrium exponentially, {\it uniformly in} $\epsilon$. 
\item[(b)] For the Parabolic regime, the initial condition on both $f$ and $\pt_x\p$ are independent of $\e$. Furthermore, when $\e$ become smaller, $f$ don't need to be near Maxwellian any more for the solution to converge to the global equilibrium exponentially.
\item[(c)]  If one directly applies the conclusion of \cite{hwang2013vlasov}, then for the high field regime, $E_h^m(0)$ and $E_\phi^m(0)$ need to be $O(\b)$ and $O(\b^3)$ initially, while for the parabolic regime, $E_h^m(0)$ and $E_\phi^m(0)$ need to be $O(1)$ and $O(\b)$ initially, see Remark \ref{explain diff} for details. Our result allows more general initial data for $f$ while keeping the optimal convergence rate at the same time, which is because of the new energy estimates we obtained in Lemma \ref{lemma:est}.
\item [(d)] {One notices that, the initial condition on  the electric field $\pt_x\p$ for the high field regime requires to be $O(\sqrt{\e})$, this is necessary because the limiting hyperbolic system won't preserve the regularity at a later time if the electric field doesn't vanishes. On the other hand, for the parabolic regime, the condition on the electric field is $O(1)$, which is because when $\e\to 0$, the VPFP system goes to a parabolic equation which enjoys better regularity compared to the high field regime. }
\item [(e)] {Notice here, although $\ll \s(t) \rl_{H^m_z}$ decays in time, the mass is still conserved, that is, $\int_\R \s(t) dx =  \int_\R \s(0) dx$ holds for all $t>0$. It is an interesting question to study the case when this conservation is not true for future research.
\item [(f)] Since $\displaystyle f = M+\sM h$ and $M$ is the global Maxwellian without randomness, so the regularity of the perturbative solution $h$ in random space implies the uniform regularity of the solution $f$. More specifically, one knows the regularity of the initial data in the random space is preserved in time. Furthermore, the bound is independent of the small parameter $\e$. }
\end{itemize}
\end{remark}

One notices that the initial condition has a bad dependency on $m$. Actually this can be eliminated by defining a new energy norm. Since the main focus of this paper is uniform regularity in $\e$, so we just give a brief proof of the following Theorem in Appendix.

\begin{theorem}
\label{new energy} 
{Define 
\begin{align}
&\ll \pt_z^lh \rl_l^2 = \ll  \frac{l+1}{l!}\pt_z^lh \rl^2\\
&\t{E}^{m,i}_h= \sum_{l\leq m}\ll \pt_x^i\pt_z^lh \rl^2_l, \qd \t{E}^m_h = \sum_{i\leq 1}\sum_{l\leq m}\ll \pt_x^i\pt_z^lh \rl^2_l = \t{E}^{m,0}_h + \t{E}^{m,1}_h,\\
&\t{E}^{m,i}_\p= \sum_{l\leq m} \ll \pt_x^i\pt_z^l\pt_x\p \rl^2_l, \qd \t{E}^m_\p=\sum_{i\leq 1}\sum_{l\leq m} \ll \pt_x^i\pt_z^l\pt_x\p \rl_l^2 = \t{E}^{m,0}_\p + \t{E}^{m,1}_\p ;
\end{align}
Theorem \ref{expo decay} still holds for the new energy norms with $A = 8\sqrt{\sum_{i=0}^\infty\frac{1}{(i+1)^2}}$.}
\end{theorem}

\begin{proof}
See Appendix \ref{proof of new energy}
\end{proof}

The proof of the main Theorem requires some equalities and inequalities, which are given below.\begin{lemma}
\label{ineqs}
Let $ \pt^{\bk} = \pt^{k_1}_z\pt^{k_2}_x$, and similar for $\pt^{\bi}, \pt^{\bl}$,
\begin{itemize}
\item [(a)] $\displaystyle \la \pt^{\bk}\pt_x\p, v\sM\pt^{\bk}h \ra = \frac{\d}{2}\pt_t\ll \pt^{\bk}\pt_x\p \rl^2$,
\item [(b)] \quad$\displaystyle \la \pt^{\bk} \pt_x\p\, \pt_v(\pt^{\bi}h), \pt^{\bl}h \ra - \frac{1}{2} \la v\pt^{\bk} \pt_x\p\, \pt^{\bi} h,  \pt^{\bl}h \ra$\\
$\displaystyle  \leq  C_1\ll \pt^{\bk}\pt_x\p \rl_{H^1_z(H^1_x)} \left(a\ll\pt^{\bi}\s \rl^2 +  2\ll \pt^{\bi} \bu\rl^2  + 2\ll (1-\Pi)\pt^{\bi}h \rl^2_\nu + \(2 + \frac{1}{a}\)\ll \pt^{\bl}\bu \rl^2+2\ll (1-\Pi) \pt^{\bl}h \rl^2_\nu \right)$,
{\item [(c)]$\displaystyle \la \pt^{\bk} \pt_x\p\, \pt_v(\pt^{\bi}h), \pt^{\bl}h \ra - \frac{1}{2} \la v\pt^{\bk} \pt_x\p\, \pt^{\bi} h,  \pt^{\bl}h \ra$\\
 $\displaystyle \leq  C_1^2  \sqrt{\ll \pt^{\bk}\pt_x\p \rl^2 + \ll \pt^{\bk}\pt_x^2\p \rl^2} \left(a\sum_{i\leq 1}\ll\pt^{\bi}\pt_z^i \s \rl^2 +  2\sum_{i\leq 1}\ll \pt^{\bi}\pt_z^i \bu\rl^2 + 2\sum_{i\leq 1}\ll (1-\Pi)\pt^{\bi}\pt_z^ih \rl^2_\nu\right.$\\
 $\displaystyle \left. + \(2+\frac{1}{a}\)\ll \pt^{\bl}\bu \rl^2 +2\ll (1-\Pi) \pt^{\bl}h \rl^2_\nu \right)$}
\item [(d)] $\displaystyle \la \pt^{\bk}\pt_x\p \, \pt_v(\pt^{\bi}h), \pt^{\bl}h \ra - \frac{1}{2} \la v\pt^{\bk}\pt_x\p \pt^{\bi}h, \pt^{\bl}h \ra$\\
$\displaystyle  \leq  C_1\ll \pt^{\bi} h\rl_{H^1_z(H^1_x)} \left(  \frac{3}{a}\ll\pt^{\bl} \bu  \rl^2 + \frac{2}{a}\ll (1-\Pi)\pt^{\bl}h\rl^2_\nu+ a\ll\pt^{\bk}\pt_x\p\rl^2 \right)$,
{\item [(e)] $\displaystyle\la \pt^{\bk} \pt_x\p\, \pt_v(\pt^{\bi}h), \pt^{\bl}h \ra - \frac{1}{2} \la v\pt^{\bk} \pt_x\p\, \pt^{\bi} h,  \pt^{\bl}h \ra$\\
$\displaystyle\leq  C_1^2\sqrt{ \ll \pt^{\bi} h\rl^2 + \ll \pt^{\bi} \pt_xh\rl^2} \(a \sum_{i\leq 1}\ll\pt^{\bk}\pt_z^i\pt_x\p\rl^2 +  \frac{3}{a}\ll\pt^{\bl} \bu  \rl^2 + \frac{2}{a}\ll (1-\Pi)\pt^{\bl}h\rl^2_\nu\).$}
\item [(f)] $\ll \pt^{\bk}\pt_x\pt_t\p\rl^2 \leq \frac{1}{\d^2} \ll \pt^{\bk}\bu \rl^2.$
\end{itemize}
where $a$ can be any positive constant.
\end{lemma}

\begin{proof}
See Appendix \ref{proof of ineqs}.
\end{proof}
\begin{remark}
Notice that in the inequalities $(b)$ and $(c)$, the dissipations of $u$ and $(1-\Pi)h$ are related to both energies $h$ and $\pt_x\p$.  However, the dissipation of $\s$ is only related to the energy of $\pt_x\p$ through $(b)$, while the dissipation of $\pt_x\p$ is only related to the energy of $h$ through $(c)$. This is why we can get the separation of the micro and macro energies in Lemma \ref{lemma:est} for $D^m_\s$ and $D^m_\p$.
\end{remark}

\section{Energy Estimates on the Microscopic Equations}
\label{sec:micro}
Now we prove the first part of Lemma \ref{lemma:est}, (\ref{eqn: micro}).
\subsection{Energy estimates for $\ll \pt_z^l h \rl^2$ }
Taking $\pt_z^l$ to (\ref{with e on t_1}), and multiplying by $\pt_z^l h$, then integrating it over $\mu(x,v,z)$, one has,
\begin{align}
& \frac{\b\d}{2}\pt_{t}\ll \pt_z^lh \rl^2 + \d\underbrace{\la \pt_z^l\pt_x\p, v\sM\pt_z^lh \ra}_{\rom{5}}  \underbrace{- \la \mL\pt_z^lh, \pt_z^lh \ra}_{\rom{6}}\nonumber\\
=&  \d \sum_{i = 0}^{l} (^l_i)\left( \underbrace{\la \pt_z^{l-i}\pt_x\p \pt_v \pt_z^i h, \pt_z^lh \ra}_{\rom{3}} -\underbrace{\frac{1}{2}\la v\pt_z^{l-i}\pt_x \p\, \pt_z^ih, \pt_z^lh\ra}_{\rom{4}} \right).
\end{align}
$\rom{5}$ and $\rom{6}$ are "good terms", since by Lemma \ref{ineqs} (a) and Proposition \ref{property of L} (d),
\begin{align}
\rom{5} = \frac{\d}{2}\pt_t\ll \pt_z^l \pt_x\p \rl^2, \quad \rom{6} \geq  \l_0\ll (1-\Pi)\pt_z^lh \rl^2_\nu +\ll \pt_z^l\bu \rl^2.
\label{good terms}
\end{align}
However, $\rom{3}$ and $\rom{4}$ are "bad terms" here, and one wants to control it by the dissapations.\\
For $i < l$, by Lemma \ref{ineqs} (c),
\begin{align}
\rom{3} - \rom{4} \leq &C_1 \ll \pt_z^ih \rl_{H^1_z(H^1_x)} \left( \frac{3}{a}\ll \pt_z^l \bu \rl^2 + \frac{2}{a}\ll (1-\Pi)\pt_z^lh \rl^2_\nu+ a\ll \pt_z^{l-i}\pt_x\p \rl^2 \right).
\end{align}
For $i = l$, by Lemma \ref{ineqs} (b),
\begin{align}
\rom{3} - \rom{4} \leq &aC_1 \ll \pt_x\p \rl_{H^1_z(H^1_x)} \ll\pt_z^l\s \rl^2 + C_1 \ll \pt_x\p \rl_{H^1_z(H^1_x)} \left(\(4+\frac{1}{a}\)\ll \pt_z^l\bu \rl^2 + 4\ll (1-\Pi) \pt_z^lh \rl^2_\nu\right).
\end{align}
Here if one treats the case of $i = l$ the same as the case of $i < l$, then the largest $i = m$ leads to $\ll \pt_z^mh\rl_{H^1_z(H^1_x)}$, which cannot be controlled by $\pt_tE^m_h$, so we treat $i = l$ differently from $i < l$. Therefore one has the energy estimate,
\begin{align}
& \frac{\d}{2}\pt_{t}\left(\b\ll \pt_z^lh \rl^2 + \d\ll \pt_z^l\pt_x\p \rl^2 \right) + \l_0 \ll (1-\Pi)\pt_z^lh \rl^2_\nu + \ll \pt_z^l\bu \rl^2\nonumber\\
\leq& C_1\d\sum_{l\neq 1, i = 0}^{l-1}(^l_i)\ll\pt_z^ih \rl_{H^1_z(H^1_x)}(\frac{3}{a}\ll \pt_z^l\bu \rl^2+ \frac{2}{a}\ll (1-\Pi)\pt_z^lh \rl^2_\nu +a\ll \pt_z^{l-i}\pt_x\p\rl^2)\nonumber\\
&+aC_1\d\sqrt{E^1_\p}\, \ll \pt_z^l\s\rl^2 +C_1\d\sqrt{E^1_\p} \(\(4+\frac{1}{a}\) \ll \pt_z^l\bu\rl^2 + 4\ll (1-\Pi)\pt_z^lh\rl^2_\nu\).
\end{align}
Summing $l$ from $0$ to $m$, one gets,
\begin{align}
&\frac{\d}{2}\pt_{t} \left[ \b E^{m,0}_h + \d E^{m,0}_\p\right] + \l_0D^{m,0}_h + D^{m,0}_u  \nonumber\\
\leq&C_1\d\sum_{l=1}^m\sum_{i = 0}^{l-1}(^l_i)\ll\pt_z^ih \rl_{H^1_z(H^1_x)}\left(\frac{3}{a} \ll \pt_z^l\bu \rl^2+ \frac{2}{a}\ll (1-\Pi)\pt_z^lh \rl^2_\nu\right) +aC_1\d\sum_{l=1}^m\sum_{i = 1}^{m}(^l_i)\ll\pt_z^{l-i}h \rl_{H^1_z(H^1_x)} \ll \pt_z^i\pt_x\p\rl^2 \nonumber\\
 &+aC_1\d\sqrt{E^1_\p} D^{m,0}_\s +C_1\d\sqrt{E^1_\p}\left(\(4+\frac{1}{a}\)D^{m,0}_u + 4D^{m,0}_h\right)  \nonumber\\
  =&C_1\d\sum_{l=1}^m\left( \sum_{i = 0}^{l-1}(^l_i)\ll\pt_z^ih \rl_{H^1_z(H^1_x)}\right)\left(\frac{3}{a}\ll \pt_z^l\bu \rl^2+  \frac{2}{a}\ll (1-\Pi)\pt_z^lh \rl^2_\nu \right) \nonumber\\
&+ aC_1\d\sum_{i = 1}^{m}\left( \sum_{l=i}^m(^l_i)\ll\pt_z^{l-i}h \rl_{H^1_z(H^1_x)}\right)\ll \pt_z^{i}\pt_x\p\rl^2\nonumber\\
  &+aC_1\d\sqrt{E^1_\p}D^{m,0}_\s  + C_1\d\sqrt{E^1_\p}\left(\(4+\frac{1}{a}\)D^{m,0}_u + 4D^{m,0}_h\right) \nonumber\\
\leq&BC_1\d\sqrt{E^m_h} \left(\frac{3}{a}D^{m,0}_u+  \frac{2}{a}D^{m,0}_h\right)+ aBC_1 \d  \sqrt{E^m_h} D^{m,0}_\p+ aC_1\d\sqrt{E^1_\p} D^{m,0}_\s \nonumber\\
&+ C_1\d\sqrt{E^1_\p}\left(\(4+\frac{1}{a}\)D^{m,0}_u+ 4D^{m,0}_h \right)\nonumber\\
\leq&\frac{BC_1\d}{a}\sqrt{E^m_h} \left(3D^{m,0}_u+  2D^{m,0}_h\right) + C_1\d\sqrt{E^1_\p}\left(\(4+\frac{1}{a}\)D^{m,0}_u+ 4D^{m,0}_h \right)\nonumber\\
& + aBC_1 \d  \sqrt{E^m_h}  D^{m,0}_\p+ aC_1\d\sqrt{E^1_\p}D^{m,0}_\s
\label{h}
\end{align}
where $B = 2\sqrt{m+1}\binom{m}{[\frac{m}{2}]}$, $[\frac{m}{2}]$ represent the smallest integer larger than $\frac{m}{2}$.

Before we move on to other estimates, let us first summarize what else we need. The goal of the energy estimates is to get an inequality like
\begin{align}
\pt_t E + D \leq \sqrt{E} D,
\end{align}
so one can use the continuity argument to get the desired estimates. Therefore, one still needs $\pt_{t} E^{m,1}_h$, $D^{m,0}_\s$, $D^{m,0}_\p$ on the LHS.

\subsection{Energy estimates for $\ll \pt_z^l\pt_xh \rl^2$}
Taking $\pt_z^l\pt_x$ to (\ref{with e on t_1}), and multiplying by $\pt_z^l \pt_xh$, then integrating it over $\mu(x,v,z)$,
\begin{align}
&\frac{\b\d}{2}\pt_{t} \ll \pt_z^{l}\pt_xh\rl^2 + \d\underbrace{\la  \pt_z^{l}\pt_x^2\p, v\sM\pt_z^l\pt_xh \ra}_{\rom{5}} \underbrace{-\la \mL\pt_z^{l}\pt_xh, \pt_z^{l}\pt_xh\ra}_{\rom{6}} \nonumber\\
=&  \d\sum_{i = 0}^{l}(^{l}_i)\la \underbrace{\pt_z^{l-i}\pt_x^2\p \pt_v\pt_z^{i}h}_{\rom{3}.1}+\underbrace{\pt_z^{l-i} \pt_x\p\pt_v\pt_x\pt_z^ih}_{\rom{3}.2}- \underbrace{\frac{v}{2}\pt_z^{l-i}\pt_x^2\p \pt_z^ih}_{\rom{4}.1} -\underbrace{\frac{v}{2}\pt_z^{l-i}\pt_x\p \pt_z^i\pt_xh}_{\rom{4}.2} , \pt_z^l\pt_xh\ra\nonumber.
\end{align}
Similar to (\ref{good terms}), for $\rom{5}$ and $\rom{6}$, one has,
\begin{align}
\rom{5} = \frac{\d}{2}\pt_t\ll \pt_z^l\pt_x^2\p \rl^2, \quad \rom{6} \geq \l_0\ll (1-\Pi)\pt_z^l\pt_xh \rl^2 + \ll \pt_z^l\pt_x\bu \rl^2.
\end{align}
For the bad terms on the RHS, 
{for $i < l$}, by Lemma \ref{ineqs} (d),
\begin{align}
\rom{3}.1 -\rom{4}.1 \leq C_1\ll \pt_z^ih \rl_{H^1_z(H^1_x)} \left(  \frac{3}{a} \ll \pt_z^l\pt_x\bu \rl^2 +  \frac{2}{a} \ll (1-\Pi)\pt_z^l\pt_xh \rl^2 +a\ll \pt_z^{l-i}\pt_x^2\p \rl^2 \right).\label{first}
\end{align}

{For $i = l$, by Lemma \ref{ineqs} (e),
\begin{align}
\rom{3}.1 -\rom{4}.1 \leq C_1^2\sqrt{\ll \pt_z^lh \rl^2 + \ll \pt_z^l\pt_xh \rl^2} \left(  \frac{3}{a} \ll \pt_z^l\pt_x\bu \rl^2 +  \frac{2}{a} \ll (1-\Pi)\pt_z^l\pt_xh \rl^2 +a\sum_{i\leq 1}\ll \pt_z^i\pt_x^2\p \rl^2 \right).
\end{align}
}

\begin{remark}
{If one treats $i = l$ the same as $i < l$, then the term $\ll \pt_z^mh \rl^2_{H^1_z(H^1_x)}$ cannot be controlled by $E^m_h$, because the term $\ll \pt_z^{m+1}\pt_xh \rl^2$ is not included in the energy term $E^m_h$. That is why we treat all these four estimates differently in (\ref{first}) - (\ref{last}).}
\end{remark}
{For $i > 0$, by Lemma \ref{ineqs} (b),}
\begin{align}
\rom{3}.2 -\rom{4}.2 \leq& C_1\ll \pt_z^{l-i}\pt_x\p \rl_{H^1_z(H^1_x)} \left( a\ll \pt_z^i\pt_x\s \rl^2 + 2 \ll \pt_z^i\pt_x\bu \rl^2 + 2 \ll(1-\Pi) \pt_z^i\pt_xh \rl^2_\nu\right.\nonumber\\
& \left. +  \(2+\frac{1}{a}\) \ll \pt_z^l\pt_x\bu \rl^2 + 2 \ll (1-\Pi)\pt_z^l\pt_xh \rl^2_\nu \right).
\end{align}

{For $i = 0$, by Lemma \ref{ineqs} (c),
\begin{align}
\rom{3}.2 -\rom{4}.2 \leq& C_1^2\sqrt{\ll \pt_z^l\pt_x\p \rl^2 + \ll \pt_z^l\pt_x^2\p \rl^2 } \left( a\sum_{i\leq 1}\ll \pt_z^i\pt_x\s \rl^2 + 2 \sum_{i\leq 1}\ll \pt_z^i\pt_x\bu \rl^2 \right.\nonumber\\
& \left.+ 2 \sum_{i\leq 1}\ll(1-\Pi) \pt_z^i\pt_xh \rl^2_\nu + \(2+\frac{1}{a}\) \ll \pt_z^l\pt_x\bu \rl^2 + 2 \ll (1-\Pi)\pt_z^l\pt_xh \rl^2_\nu \right).\label{last}
\end{align}
}

Combining all the terms gives,
Summing $l$ from $0$ to $m$ gives,
{
\begin{align}
&\frac{\d}{2}\pt_{t} \left[\b E^{m,1}_h + \d    E^{m,1}_\p \right] + \l_0D^{m,1}_h + D^{m,1}_u
\nonumber\\
\leq& C_1\d\sum_{l = 1}^m\( \sum_{i = 0}^{l-1}(^{l}_i)\ll \pt_z^ih\rl_{H^1_z(H^1_x)}\)\(\frac{3}{a}\ll\pt_z^{l}\pt_x\bu\rl^2 +\frac{2}{a}\ll (1-\Pi)\pt_z^{l}\pt_xh\rl^2_\nu + a\ll \pt_z^{l-i}\pt_x^2\p \rl^2 \) \nonumber\\
&+ C_1\d\sum_{l = 1}^m\( \sum_{i = 1}^{l}(^{l}_i)\ll \pt_z^{l-i} \pt_x\p\rl_{H^1_z(H^1_x)}\)\left( a\ll \pt_z^i\pt_x\s \rl^2 + 2 \ll \pt_z^i\pt_x\bu \rl^2 + 2 \ll(1-\Pi) \pt_z^i\pt_xh \rl^2_\nu\right.\nonumber\\
& \left. +  \(2+\frac{1}{a}\) \ll \pt_z^l\pt_x\bu \rl^2 + 2 \ll (1-\Pi)\pt_z^l\pt_xh \rl^2_\nu \right)\nonumber\\
&+ C_1^2\d\sum_{l = 0}^m\sqrt{\ll \pt_z^lh \rl^2 + \ll \pt_z^l\pt_xh \rl^2} \left(  \frac{3}{a} \ll \pt_z^l\pt_x\bu \rl^2 +  \frac{2}{a} \ll (1-\Pi)\pt_z^l\pt_xh \rl^2 +a\sum_{i\leq 1}\ll \pt_z^i\pt_x^2\p \rl^2 \right)\nonumber\\
&+C_1^2\d\sum_{l = 0}^m\sqrt{\ll \pt_z^l\pt_x\p \rl^2 + \ll \pt_z^l\pt_x^2\p \rl^2 } \left( a\sum_{i\leq 1}\ll \pt_z^i\pt_x\s \rl^2 + 2 \sum_{i\leq 1}\ll \pt_z^i\pt_x\bu \rl^2 \right.\nonumber\\
& \left.+ 2 \sum_{i\leq 1}\ll(1-\Pi) \pt_z^i\pt_xh \rl^2_\nu + \(2+\frac{1}{a}\) \ll \pt_z^l\pt_x\bu \rl^2 + 2 \ll (1-\Pi)\pt_z^l\pt_xh \rl^2_\nu \right)\nonumber\\
\leq & BC_1\d \sqrt{E^m_h}\(\frac{3}{a}D^{m,1}_u + \frac{2}{a}D^{m,1}_h + a D^{m,1}_\p\) \nonumber\\
&+ BC_1\d\sqrt{E^m_\p}\(aD^{m,1}_\s + 2D^{m,1}_u + 2D^{m,1}_h + \(2+\frac{1}{a}\)D^{m,1}_u + 2D^{m,1}_h\)\nonumber\\
&+C_1^2\d\sqrt{E^m_h}\(\frac{3}{a}D^{m,1}_u + \frac{2}{a}D^{m,1}_h\) + aC_1^2\d\sqrt{m+1}\sqrt{E^m_h}D^{1,1}_\p \nonumber\\
&+C_1^2\d\sqrt{m+1}\sqrt{E^m_\p}\( aD^{1,1}_\s + 2D^{1,1}_u + 2D^{1,1}_h\) + C_1^2\d\sqrt{E^m_\p}\( \(2+\frac{1}{a}\)D^{m,1}_u + 2D^{m,1}_h \)\nonumber\\
\leq &\frac{(B+1)C_1^2\d}{a}\sqrt{E^m_h} \left(3D^{m,1}_u + 2D^{m,1}_h \right) + \(B+1\)C_1^2\d\sqrt{E^m_\p} \left(\(4+\frac{1}{a}\)D^{m,1}_u + 4D^{m,1}_h \right)  \nonumber\\
 &+a\(B+\sqrt{m+1}\)C_1^2\d\sqrt{E^m_h}D^{m,1}_\p +a\(B+\sqrt{m+1}\)C_1^2\d\sqrt{E^m_\p}D^{m,1}_\s,
 \label{pt_xh}
\end{align}
}
where $A$ is defined as (\ref{def of A}). Now combining (\ref{h}) and (\ref{pt_xh}) completes the energy estimates for the microscopic system,
\begin{align}
&\frac{\d}{2}\pt_{t} \left[ \b \underbrace{E^m_h}_{\rom{1}} + \d \underbrace{E^m_\p}_{\rom{5}}\right] + \underbrace{\l_0D^m_h +D^m_\bu}_{\rom{6}}\nonumber\\
\leq& \underbrace{\frac{AC_1^2\d}{a}\sqrt{E^m_h} \left(3D^m_\bu +2D^m_h\right) + 2AC_1^2\d\sqrt{E^m_\p}  \left(\(4+\frac{1}{a}\)D^m_\bu +4D^m_h\right)}_{\rom{3}+\rom{4}}\nonumber\\
&\underbrace{+aAC_1^2\d\sqrt{E^m_h}D^m_\p + aAC_1^2\d\sqrt{E^m_\p}D^m_\s}_{\rom{3}+\rom{4}}.
\end{align}

Up to now, one still needs the  dissipations $D^m_\s$ and $D^m_\p$ on LHS to balance the bad terms on RHS. So next we turn to the  macroscopic system.

\section{Energy Estimates on the Macroscopic System}
\label{sec:macro}
We now prove (\ref{eqn: macro}) in Lemma \ref{lemma:est}.
\subsection{Dissipation terms $\ll \pt_z^l\s \rl^2$ and $\ll\pt_z^l\pt_x\p \rl^2$}
Taking $\pt^l_z$ to (\ref{eqn for u with e_2}) and multiplying by $\pt_z^l\pt_x\p$, then integrating it over $\mu(x,z)$, one has,
\begin{align}
&\b\d \underbrace{\la \pt_{t} \pt_z^l\bu,  \pt_z^l\pt_x\p\ra}_{\rom{1}}+ \b\underbrace{\la \pt_z^l\pt_x\s, \pt_z^l\pt_x\p\ra}_{\rom{2}.1} + \underbrace{\la \pt^l_z\bu, \pt_z^l\pt_x\p\ra}_{\rom{6}} +\d\underbrace{\ll \pt_z^l\pt_x\p\rl^2}_{\rom{5}}\nonumber\\
=& -\b\underbrace{ \la (1-\Pi)\pt^l_z\pt_xh , v^2\sM\pt_z^l\pt_x\p\ra}_{\rom{2}.2} - \d\sum_{i=0}^l(^l_i)\underbrace{\la \pt_z^{l-i}\pt_x\p\pt_z^i\s, \pt_z^l\pt_x\p\ra}_{\rom{3}} .
\label{pc_111}
\end{align}
First one has,
\begin{align}
&\rom{1}= \pt_{t}  \la \pt_z^l\bu, \pt_z^l\pt_x\p \ra  - \la \pt_z^l\bu,  \pt_{t}\pt_z^l \pt_x\p\ra,
\end{align}
then by Lemma \ref{ineqs} (d),
\begin{align}
&\la \pt_z^l\bu, \pt_z^l\pt_x\pt_t\p \ra = \d\la \pt_z^l\pt_t\s, \pt_z^l\pt_t\p\ra = \d\ll\pt_z^l\pt_x \pt_t\p\rl^2 \leq \frac{1}{\d}\ll \pt_z^l\bu \rl^2.
\end{align}
$\rom{2}.1$ and $\rom{6}$ are "good terms" here, since
\begin{align}
&\rom{2}.1 =\la  \pt_z^l\s, -\pt_z^l\pt_x^2\p\ra = \ll \pt_z^l\s \rl^2,\\
&\rom{6} = - \la  \pt_z^l\pt_x\bu,  \pt_z^l\p\ra = \d\la \pt_z^l\pt_t\s, \pt_z^l\p\ra = \d\la \pt_z^l\pt_x\pt_t\p, \pt_z^l \pt_x\p \ra =  \frac{\d}{2}\pt_{t}\ll \pt_z^l\pt_x\p \rl^2,
\end{align}
while $\rom{2}.2$ and $\rom{3}$ are "bad terms",
\begin{align}
- \rom{2}.2  =& \la (1-\Pi)\pt_z^lh, v^2\sM\pt_z^l\pt_x^2\p\ra \leq  \frac{1}{2}\ll v\sM\pt_z^l \s \rl^2+ \frac{1}{2} \ll v(1-\Pi)\pt_z^lh \rl^2 \nonumber\\
&\leq \frac{1}{2}\ll\pt_z^l \s \rl^2+ \frac{1}{2} \ll (1-\Pi)\pt_z^lh \rl^2_\nu.
\end{align}
Note, for $l = 0$, 
\begin{align}
    -\rom{3} = \la \pt_x\p\s, \pt_x\p \ra = -\la \pt_x^2\p, (\pt_x\p)^2\ra = \la \pt_x\p, 2\pt_x\p \pt_x^2\p \ra = -2\la \pt_x\p \s, \pt_x\p \ra,
\end{align}
which implies, 
\begin{align}
    -\rom{3} = 0.
\end{align}
For $l>0$, and $i = 0$, 
\begin{align}
 -\rom{3} =&  \la \s, \(\pt_z^l\pt_x\p\)^2\ra= \la \pt_x\p, \pt_x\(\pt_z^l\pt_x\p\)^2 \ra = -2\la \pt_x\p\pt_z^l\s, \pt_z^l\pt_x\p \ra\nonumber\\
 \leq& C_1\ll \pt_x\p \rl_{H^1_z(H^1_x)}\(\ll \pt_z^l\s\rl^2  + \ll \pt_z^l\pt_x\p \rl^2\),
\end{align}
and for $0< i \leq l$,
\begin{align}
-\rom{3} \leq& \frac{C_1}{2}\ll \pt_z^{l-i}\pt_x\p \rl_{H^1_z(H^1_x)}\( \ll \pt_z^i\s \rl^2 + \ll \pt_z^l\pt_x\p \rl^2\).
\end{align}
Combining all terms in (\ref{pc_111}), one has,
\begin{align}
&\d\pt_{t}\left[ \b\la \pt_z^l \bu, \pt_z^l\pt_x\p \ra  + \frac{1}{2}\ll \pt_z^l\pt_x\p \rl^2\right] + \frac{\b}{2}\ll\pt_z^l \s \rl^2 +  \d\ll \pt_z^l\pt_x\p\rl^2\nonumber\\
 \leq&\b\ll \pt_z^l\bu \rl^2 + \frac{\b}{2}\ll (1-\Pi)\pt_z^l h \rl^2_\nu + 2C_1\d\sum_{l\neq 0, i=1}^{l}(^l_i)\ll \pt_z^{l-i}\pt_x\p\rl_{H^1_z(H^1_x)} \(\ll \pt_z^{i}\s \rl^2 + \ll\pt_z^l\pt_x\p\rl^2\).
  \end{align}
Summing $l$ from $0$ to $m$ gives,
\begin{align}
&\d\pt_{t}\left[ \b\sum_{l=0}^m\la \pt_z^l \bu, \pt_z^l\pt_x\p \ra  + \frac{1}{2}E^{m,0}_\p\right] + \frac{\b}{2}D^{m,0}_\s +  \d D^{m,0}_\p\nonumber\\
  \leq &\b D^{m,0}_u + \frac{\b}{2}D^{m,0}_h + 2AC_1\d\sqrt{E^m_\p} \(D^{m,0}_\s + D^{m,0}_\p\).
    \label{macro_1}
  \end{align}

\subsection{Dissipation terms $\ll \pt_z^l\pt_x\s \rl^2$ and $\ll\pt_z^l\pt_x^2\p \rl^2$ }
Taking $\pt^{l}_z$ to (\ref{eqn for u with e_1}) and multiplying by $\pt_z^{l}\pt_x\s$, then integrating it over $\mu(x,z)$,
\begin{align}
&\b\d\underbrace{\la \pt_{t} \pt_z^{l}\bu, \pt_z^{l}\pt_x\s\ra }_{\rom{1}}+\b\underbrace{\ll \pt_z^{l}\pt_x\s \rl^2}_{\rom{2}.1} + \underbrace{\la \pt^{l}_z\bu, \pt_z^{l}\pt_x\s\ra }_{\rom{6}} + \d\underbrace{\la \pt_z^{l}\pt_x\p, \pt_z^{l}\pt_x\s\ra}_{\rom{5}} \nonumber\\
=&- \b\underbrace{ \la (1-\Pi)\pt_x\pt^{l}_zh, v^2\sM\pt_z^{l}\pt_x\s \ra}_{\rom{2}.2} - \d\sum_{i=0}^{l}(^{l}_i)\underbrace{\la \pt_z^{l-i}\pt_x\p\pt_z^i\s,\pt_z^{l}\pt_x\s\ra }_{\rom{3}}.
\label{pc_222}
\end{align}
Note that,
\begin{align}
\rom{1} &= \pt_{t} \la \pt_z^{l}\bu, \pt_x\pt_z^{l}\s \ra - \la\pt_z^{l}\bu, \pt_z^{l}\pt_x\pt_t\s \ra = \pt_{t} \la \pt_z^{l}\pt_x\bu, \pt_z^{l}\pt_x^2\p \ra - \frac{1}{\d}\ll  \pt_z^{l}\pt_x\bu\rl^2,\label{number1}\\
\rom{6} &=\la \pt_z^{l}\bu, \pt_z^{l}\pt_x\s  \ra = \d\la \pt_z^{l}\pt_t\s, \pt_z^{l}\s  \ra = \frac{\d}{2}\pt_{t} \ll  \pt_z^{l}\s \rl^2 = \frac{\d}{2}\pt_{t} \ll  \pt_z^{l}\pt_x^2\p \rl^2,\\
\rom{5} &= -\la \pt_z^{l}\pt_x^2\p,  \pt_z^{l}\s \ra =\ll \pt_z^{l}\pt_x^2\p\rl^2, \\
-\rom{2}.2 &\leq   \frac{1}{2}\ll \pt_z^{l}\pt_x\s \rl^2 + \frac{1}{2}\ll  (1-\Pi)\pt_z^{l}\pt_xh\rl^2_\nu, \\
\text{For }i &\neq 0\nonumber\\
-\rom{3} &\leq \frac{C_1}{2} \ll \pt_z^{l-i}\pt_x\p \rl_{H^1_z(H^1_x)}( \ll \pt_z^i\s \rl^2+ \ll \pt_z^{l}\pt_x\s \rl^2 ),\\
\text{For }i &= 0\nonumber\\
-\rom{3} &\leq \frac{C_1^2}{2} \sqrt{\ll \pt_z^l\pt_x\p \rl^2 + \ll \pt_z^l\pt_x^2\p \rl^2}\( \sum_{i\leq 1}\ll \pt_z^i\s \rl^2 + \ll \pt_z^{l}\pt_x\s \rl^2 \)\label{number last}.
\end{align}
Using (\ref{number1}) - (\ref{number last}) in (\ref{pc_222}) implies,
\begin{align}
 &\d\pt_{t}\left[\b \la \pt_z^{l}\pt_x\bu, \pt_z^{l}\pt_x^2\p \ra + \frac{1}{2} \ll  \pt_z^{l}\pt_x^2\p \rl^2 \right] + \frac{ \b}{2}\ll \pt_z^{l}\pt_x\s \rl^2 +\d \ll \pt_z^{l}\pt_x^2\p\rl^2  \nonumber\\
 \leq&\frac{\b}{2} \ll  (1-\Pi)\pt_z^{l}\pt_xh\rl^2_\nu + \b\ll  \pt_z^{l}\pt_x\bu\rl^2 + \frac{C_1\d}{2 }\sum_{i=1}^{l}(^{l}_i) \ll \pt_z^{l-i}\pt_x\p \rl_{H^1_z(H^1_x)}( \ll \pt_z^i\s \rl^2+ \ll \pt_z^{l}\pt_x\s \rl^2 ) \nonumber\\
 &+\frac{C_1^2}{2} \sqrt{\ll \pt_z^l\pt_x\p \rl^2 + \ll \pt_z^l\pt_x^2\p \rl^2}\( \sum_{i\leq 1}\ll \pt_z^i\s \rl^2 + \ll \pt_z^{l}\pt_x\s \rl^2 \).
\end{align}
Summing $l$ from $0$ to $m-1$, one has,
\begin{align}
 &\d\pt_{t}\left[\b \sum_{l=0}^{m}\la  \pt_z^{l}\pt_x\bu, \pt_z^{l}\pt_x^2\p \ra + \frac{1}{2} E^{m,1}_\p \right] + \frac{ \b}{2} D^{m,1}_\s + \d D^{m,1}_\p  \nonumber\\
  \leq &\b D^{m,1}_u +\frac{\b}{2}D^{m,1}_h +AC_1^2\d\sqrt{E^m_\p} D^m_\s.
  \label{macro_2}
\end{align}
Combining (\ref{macro_1}) and (\ref{macro_2}), one finishes the energy estimates for the microscopic system,
\begin{align}
 &\d\pt_{t}\left[\underbrace{\b \sum_{l=0}^{m}\la  \pt_z^{l}\bu, \pt_z^{l}\pt_x\p \ra + \b\sum_{l=0}^m\la \pt_z^l\pt_x \bu, \pt_z^l\pt_x^2\p \ra}_{\rom{1}} + \underbrace{\frac{1}{2} E^m_\p}_{\rom{6}} \right] + \frac{ \b}{2}\underbrace{D^m_\s}_{\rom{2}.1} + \d \underbrace{D^m_\p}_{\rom{5}}  \nonumber\\
  \leq &\b \underbrace{D^m_\bu}_{\rom{1}}+ \frac{\b}{2}\underbrace{D^m_h}_{\rom{2}.2}   + \underbrace{AC_1^2\d\sqrt{E^m_\p} \( 3D^m_\s + 2D^m_\p\)}_{\rom{3}}.
  \end{align}

\section{Exponential Decay to the Maxwellian}
\label{sec: expo decay}
Before we do the analysis for the two energy estimates, we first go through the process in a more general framework. If one has the  energy estimate,
\begin{align}
\frac{1}{2}\pt_t \hat{E} + \alpha D \leq \beta \sqrt{\hat{E}}D,
\label{general}
\end{align}
and one wants to get an exponential decay for $E$, then one requires,
\begin{align}
&\text{REQUIREMENT 1: }\quad \hat{E}  \sim E \leq D .
\end{align}
On the other hand, one needs the dissipations on the LHS to balance the "bad terms" on the  RHS, so one requires,
\begin{align}
&\text{REQUIREMENT 2: }\quad \a > 0.
\end{align}
Since (\ref{general}) is equivalent to,
\begin{align}
\pt_t \sqrt{\hat{E}} \apprle& \frac{1}{\sqrt{\hat{E}}}\left(\beta\sqrt{\hat{E}} -\alpha\right)D,
\end{align}
therefore,  if one assumes the initial data satisfies,
\begin{align}
{\frac{1}{\sqrt{\hat{E}}}\left(\beta\sqrt{\hat{E}} -\alpha\right)D \apprle -\frac{\alpha}{2\sqrt{\hat{E}}}D, \quad\text{or equivalently, }\quad
\sqrt{\hat{E}(0)} \leq O\left(\frac{\alpha}{2\beta}\right) ,}
\end{align}
then by standard continuity argument, since $\sqrt{\hat{E}}$ is decreasing, so for $t>0$,
\begin{align}
\pt_t \sqrt{\hat{E}} \apprle -\frac{\a}{2\sqrt{\hat{E}}}D,
\label{333}
\end{align}
and $D\geq \hat{E}$,  (\ref{333}) implies the exponential decay,
\begin{align}
 \hat{E} \apprle e^{-C\alpha t}\hat{E}(0)\quad \sim \quad E(t) \apprle e^{-C\alpha t}E(0).
\end{align}
Furthermore, if one wants to get the optimal convergence rate with least restriction on initial data, then one needs,
\begin{align}
\text{REQUIREMENT 3: }\quad &\sqrt{\hat{E}(0)} \leq O\(\frac{\a}{\beta}\) \text{ independent of small parameters.}\nonumber\\
\text{REQUIREMENT 4: }\quad &\alpha \text{ should be as large as possible. }
\end{align}

\begin{remark}
\label{explain diff}
Without uncertainty, if one directly uses the energy estimates from \cite{hwang2013vlasov}, for the high field regime, where $\d = 1$ , then when the small parameter $\e$ are put in, the energy estimates become,
\begin{align}
&\frac{1}{2}\pt_{t} \left[  E^m_h + \frac{1}{\e}E^m_\p\right] + \frac{1}{\e}\(D^m_h +D^m_\bu\)\nonumber\\
\apprle & \frac{1}{\e}\sqrt{E^m_h + E^m_\p}(D^m_\bu+D^m_h + D^m_\p +D^m_\s),
\label{micro_1}
\end{align}
and,
\begin{align}
 &\pt_{t}\left[ \sum_{l=0}^{m}\la  \pt_z^{l}\bu, \pt_z^{l}\pt_x\p \ra + \sum_{l=0}^m\la \pt_z^l \pt_x\bu, \pt_z^l\pt_x^2\p \ra + \frac{1}{2\e} E^m_\p \right] + \frac{1 }{2}D^m_\s + \frac{1}{\e} D^m_\p  \nonumber\\
  \apprle & D^m_\bu + D^m_h  +\frac{1}{\e}\sqrt{E^m_h+E^m_\p}( D^m_\p +D^m_\s).
  \label{macro_22}
  \end{align}
 Let
$G^m =  \sum_{l=0}^{m-1}\la  \pt_z^{l}\bu, \pt_z^{l}\pt_x\p \ra + \sum_{l=0}^m\la \pt_z^l \pt_x\bu, \pt_z^l\pt_x^2\p \ra + \frac{1}{2\e} E^m_\p.
$
 Since
$-\e E^m_h + \frac{1}{4\e}E^m_\p \leq G^m\leq \e E^m_h + \frac{3}{4\e}E^m_\p,
$
so if one  combines the microscopic and macroscopic energy estimates (\ref{micro_1}) + $\g$ (\ref{macro_22}), one needs $\g \leq O(\frac{1}{\e})$ to satisfy REQUIREMENT 1. Furthermore, if one wants to get the optimal convergence rate based on this energy estimate, then one needs the dissipation terms to be as large as possible, that is $\g$ as large as possible, which means $\g = O(\frac{1}{\e})$. Therefore one derives,
\begin{align}
&\frac{1}{2}\pt_t \hat{E}^m + \frac{1}{\e}(D^m_h  + D^m_\bu) + \frac{1}{\e}D^m_\s + \frac{1}{\e^2}D^m_\p \nonumber\\
\apprle& \frac{1}{\e}\sqrt{E^m_h + E^m_\p} (D^m_\bu+D^m_h) +\frac{1}{\e^2}\sqrt{E^m_h + E^m_\p}D^m_\s+ \frac{1}{\e^2}\sqrt{E^m_h + E^m_\p}D^m_\p ,
\label{comb_1}
\end{align}
where
$\hat{E}^m = \(E^m_h + \frac{1}{\e}E^m_\p\) + \frac{1}{\e}G^m \sim E^m_h + \frac{1}{\e^2}E^m_\p.
$
 So (\ref{comb_1}) leads to,
\begin{align}
\frac{1}{2}\pt_t \hat{E}^m +  \frac{1}{\e}(D^m_h  + D^m_\bu) + \frac{1}{\e}D^m_\s + \frac{1}{\e^2}D^m_\p  \leq \frac{1}{\e}\sqrt{\hat{E}^m} (D^m_h  + D^m_\bu) + \frac{1}{\e^2}\sqrt{\hat{E}^m} D^m_\s  +\frac{1}{\e^2}\sqrt{\hat{E}^m}  D^m_\bu.
\label{result_1}
\end{align}
Compare the term $D^m_\s$, one notes that $\sqrt{\hat{E}^m}$ needs to be $O(\e)$ such that the bad term on the RHS can be controlled by the $O(1)$ dissipation on the LHS. That is one requires
\begin{align}
E^m_h(0) + \frac{1}{\e^2}E^m_\p (0) \apprle O( \e)
\end{align}
to obtain the exponential decay
\begin{align}
E^m_h + \frac{1}{\e}E^m_\p \apprle e^{-O(1)t}\( E^m_h (0)+ \frac{1}{\e}E^m_\p\).
\end{align}
This means the initial data $E^m_h = O(\e)$, $E^m_\p  = O(\e^3)$.
These conditions are much stronger than the one in  (\ref{initial cond_1}) of Theorem \ref{expo decay}.

However, if the coefficient of the $D_\s$ only depends on $E^m_\p$, like the estimates we obtained in Lemma \ref{lemma:est}, then (\ref{comb_1}) becomes,
\begin{align}
&\frac{1}{2}\pt_t \hat{E}^m + \frac{1}{\e}(D^m_h  + D^m_\bu) + \frac{1}{\e}D^m_\s + \frac{1}{\e^2}D^m_\p \nonumber\\
\apprle& \frac{1}{\e}\sqrt{E^m_h + E^m_\p} (D^m_\bu+D^m_h) +\frac{1}{\e}\sqrt{\frac{1}{\e^2}E^m_\p}D^m_\s+ \frac{1}{\e^2}\sqrt{E^m_h}D^m_\p \nonumber\\
\apprle &\frac{1}{\e}\sqrt{\hat{E}^m} (D^m_h  + D^m_\bu) + \frac{1}{\e}\sqrt{\hat{E}^m} D^m_\s  +\frac{1}{\e^2}\sqrt{\hat{E}^m}  D^m_\bu.
\end{align}
Now the bad terms and good terms can be well balanced even if the initial data of $\hat{E}^m$ is $O(1)$.
\end{remark}

\subsection{The high field regime}
For the high field regime, where $\d = 1$, set
\begin{align}
&F^m = \b E^m_h +   E^m_\p , \quad G^m =\e \sum_{l=0}^{m-1}\la  \pt_z^{l}\bu, \pt_z^{l}\pt_x\p \ra + \b\sum_{l=0}^m\la \pt_z^l \pt_x\bu, \pt_z^l\pt_x^2\p \ra + \frac{1}{2} E^m_\p, \nonumber\\
&\hat{E}^m = F^m + \frac{2\l_0}{\e} G^m, \quad E^m = \e E^m_h +  \frac{1}{\e}E^m_\p, \qd a = 1 \label{rescale_hf}
\end{align}
where $F^m$ is the term inside $\pt_t$ in (\ref{eqn: micro}) and $G^m$ is that in  (\ref{eqn: macro}).

{By (\ref{u s h}) and Young's Inequality}, one can bound $G^m$ by
\begin{align}
 -\e^2 E^m_h + (\frac{1}{2} - \frac{1}{4}) E^m_\p \leq &G^m \leq \e^2E^m_h +  (\frac{1}{2} + \frac{1}{4}) E^m_\p, \nonumber\\
  -\b^{2} E^m_h +\frac{1}{4}E^m_\p \leq &G^m \leq  \b^{2}E^m_h+ \frac{3}{4} E^m_\p.
  \label{bounds of G}
 \end{align}
Since $\l_0 \leq \frac{1}{4}$, thus one obtains,
\begin{align}
 (1 - 2\l_0)\e E^m_h +( \e + \frac{\l_0}{2}) \frac{1}{\e}E^m_\p \leq& \hat{E}^m \leq \( 1 + 2\l_0 \)\e E^m_h + \( \e + \frac{3\l_0}{2}\)\frac{1}{\e} E^m_\p,\nonumber\\
\frac{\l_0}{2}E^m\leq& \hat{E}^m \leq \frac{3}{2} E^m,
\label{ineq of E}
\end{align}
or equivalently,
\begin{align}
 \frac{3}{2}\sqrt{\hat{E}^m}  \leq  \sqrt{\e E^m_h} &+   \sqrt{\frac{1}{\e}E^m_\p} \leq \frac{2}{\l_0}\sqrt{\hat{E}^m}.
\end{align}
So one has the equivalence between the energies $E^m$ and $\hat{E}^m$, {besides, the dissipation terms can be lower bounded by $E^m$,}
\begin{align}
 \hat{E}^m\sim E^m\leq \e \(D^m_\bu +  D^m_h +  D^m_\s \)+ \frac{1}{\e} D^m_\p.
 \label{612}
\end{align}

By  (\ref{eqn: micro}) + $\frac{\l_0}{\e}$(\ref{eqn: macro}), one has the
energy estimates,
\begin{align}
&\frac{1}{2}\pt_{t}\hat{E}^m + \left(\l_0 - \frac{\l_0}{2}\right)D^m_h +\left(1 - \l_0\right) D^m_\bu + \frac{\l_0}{2}D^m_\s + \frac{\l_0}{\e} D^m_\p\nonumber\\
\leq& AC_1^2\left[ \left( \frac{1}{\sqrt{\e}}\sqrt{\e E^m_h} + 2 \sqrt{\e}\sqrt{\frac{1}{\e}E^m_\p} \right)(5D^m_\bu + 4D^m_h ) + \sqrt{\e}\(1+\frac{1}{\b}\)\sqrt{\frac{1}{\e}E^m_\p} D^m_\s + \frac{1}{\sqrt{\e}}\(1+\frac{1}{\b}\)\sqrt{\e E^m_h} D^m_\p  \right]\nonumber\\
\leq & \frac{10}{\sqrt{\b}}AC_1^2\(\frac{2}{\l_0}\sqrt{\hat{E}^m}\) (D^m_h + D^m_\bu )  + \frac{2}{\sqrt{\e}}AC_1^2\(\frac{2}{\l_0}\sqrt{\hat{E}^m}\)D^m_\s + \frac{2}{\e^{3/2}}AC_1^2\(\frac{2}{\l_0}\sqrt{\hat{E}^m}\)D^m_\p,
\end{align}
which implies
\begin{align}
&\frac{1}{2}\pt_{t}\hat{E}^m + \frac{\l_0}{2} \(D^m_h + D^m_\bu\) + \frac{\l_0}{2}D^m_\s + \frac{\l_0}{\e} D^m_\p\nonumber\\
\leq & \frac{20AC_1^2}{\l_0\sqrt{\b}}\sqrt{\hat{E}^m}(D^m_h + D^m_\bu )  + \frac{4AC_1^2}{\l_0\sqrt{\e}}\sqrt{\hat{E}^m}D^m_\s + \frac{4AC_1^2}{\l_0\e^{3/2}}\sqrt{\hat{E}^m}D^m_\psi.
\end{align}

Therefore, by standard continuity argument, under the condition of,
\begin{align}
&\sqrt{\hat{E}^m(0)} \leq \min\left\{  \frac{ \frac{\l_0}{4}    }{\frac{20AC_1^2}{\l_0\b^{1/2}}},\   \frac{ \frac{\l_0}{4}    }{\frac{4AC_1^2}{\l_0\b^{1/2}}}, \  \frac{ \frac{\l_0}{2\e}    }{\frac{4AC_1^2}{\l_0\e^{3/2}}}\right\},\nonumber
\end{align}
which holds if,
\begin{align}
&\hat{E}^m(0) \leq\left(  \frac{\l_0^2\e^{1/2}}{80AC_1^2}\right)^2,\nonumber
\end{align}
or equivalently,
\begin{align}
& E^m_h(0) + \frac{1}{\b^2}E^m_\p(0) \leq  \frac{2\l_0^3}{\(80AC_1^2\)^2},
\end{align}
one then has the estimate,
\begin{align}
& \frac{1}{2}\pt_t\hat{E}^m + \frac{\l_0}{4}(D^m_h + D^m_\bu + D^m_\s + \frac{1}{\e}D^m_\p) \leq 0.
\label{ineq of estimates}
\end{align}
which implies,
\begin{align}
& \frac{1}{2}\hat{E}^m(t) - \frac{1}{2}\hat{E}^m(0)\leq -\frac{\l_0}{4} \int_0^t\( E^m_h(s) + \frac{1}{\e}E^m_\p(s)\) ds \nonumber\\
 & \frac{\l_0}{4}\( \e E^m_h(t)+ \frac{1}{\e}E^m_\p(t) \) \leq -\frac{\l_0}{4} \int_0^t\( E^m_h(s) + \frac{1}{\e}E^m_\p(s)\) ds + \frac{3}{4}\(\e E^m_h(0) + \frac{1}{\e}E^m_\p(0)\) \nonumber\\
 &  E^m_h(t) \leq -\frac{1}{\e} \int_0^t E^m_h(s) ds +  \frac{3}{\e\l_0}\(\e E^m_h(0) + \frac{1}{\e}E^m_\p(0)\)\nonumber\\
 & E^m_h(t) \leq \frac{3}{\l_0}e^{-\frac{t}{\e}}\( E^m_h(0) + \frac{1}{\e^2}E^m_\p(0)\).
\end{align}
Similarly, for $E^m_\p(t)$, 
\begin{align}
& E^m_\p(t)  \leq - \int_0^tE^m_\p(s) ds + \frac{3\e}{\l_0}\(\e E^m_h(0) + \frac{1}{\e}E^m_\p(0)\) \nonumber\\
&E^m_\p(t) \leq \frac{3}{\l_0}e^{-t}\(\e^2 E^m_h(0) + E^m_\p(0)\).\label{ineq of estimates_2}
\end{align}
This completes the proof of (\ref{conclusion: HFL}) in Theorem \ref{expo decay}.

\subsection{The parabolic regime}
For the parabolic regime, where $\d = \b$, set
 \begin{align}
 &F^m = \b E^m_h + \b  E^m_\p , \quad G^m =\b \sum_{l=0}^{m-1}\la  \pt_z^{l}\bu, \pt_z^{l}\pt_x\p \ra + \b \sum_{l=0}^m\la \pt_z^l \pt_x\bu, \pt_z^l\pt_x^2\p \ra + \frac{1}{2} E^m_\p ,\\
&\hat{E}^m = F^m + 2\l_0 G^m, \quad E^m =  \b E^m_h +E^m_\p, \qd a = \sqrt{\e}.
\end{align}
Similar to (\ref{bounds of G}), the bounds of $G^m$ is,
\begin{align}
 -\b E^m_h + \frac{1}{4} E^m_\p\leq -\e E^m_h + \(\frac{1}{2} - \frac{\e}{4}\)E^m_\p \leq &G^m\leq \e E^m_h + \(\frac{1}{2} + \frac{\e}{4}\)E^m_\p \leq \b E^m_h+\frac{3}{4} E^m_\p,
 \end{align}
and $\l_0 \leq \frac{1}{4}$, so one obtains,
\begin{align}
 (1-2\l_0)\e E^m_h + (\e + \frac{\l_0}{2})E^m_\p \leq& \hat{E}^m\leq (1+2\l_0)\e E^m_h + (\e + \frac{3\l_0}{2}) ,\nonumber\\
\frac{\l_0}{2}E^m \leq &\hat{E}^m\leq \frac{3}{2}E^m,\label{ineq of E_2}
\end{align}
or equivalently,
\begin{align}
\frac{2}{3}\sqrt{\hat{E}^m} \leq \e^{1/2}\sqrt{E^m_h} &+ \sqrt{E^m_\p}\leq\frac{2}{\l_0}\sqrt{\hat{E}^m}.
\end{align}
By  (\ref{eqn: micro}) + $\l_0$(\ref{eqn: macro}), one has energy estimates,
\begin{align}
&\frac{1}{2}\pt_t\hat{E}^m +\( \frac{\l_0}{\e} - \frac{\l_0}{2}\)D^m_h+  \(\frac{1}{\e} - \l_0\)D^m_\bu    + \frac{\l_0}{2} D^m_\s  + \l_0  D^m_\p \nonumber\\
\leq& \frac{AC_1^2}{\sqrt{\e}}\(\sqrt{E_h^m} + 2\sqrt{E_\p^m}\)\(5D^m_\bu + 4D^m_h\)+ AC_1^2\(1+\l_0\) \sqrt{E^m_\p}  D^m_\s+ \sqrt{\e} AC_1^2 \sqrt{E^m_h} D^m_\p+2\l_0AC_1^2 \sqrt{E^m_\p} D^m_\p \nonumber\\
\leq & 10AC_1^2\( \frac{1}{\e^{1/2}}\frac{2}{\l_0}\sqrt{\hat{E}^m}\) \(D^m_\bu + D^m_h\) + 2AC_1^2\( \frac{2}{\l_0}\sqrt{\hat{E}^m} \) D^m_\s+ 2AC_1^2\(\frac{2}{\l_0}\sqrt{\hat{E}^m} \) D^m_\p,
\end{align}
which implies,
\begin{align}
&\frac{1}{2}\pt_t\hat{E}^m + \frac{\l_0}{2\b^2} \left(  \b D^m_h + \b D^m_\bu) \right)   + \frac{\l_0}{2\b}\left(\b D^m_\s\right)  + \l_0  D^m_\p\nonumber\\
\leq&\frac{20AC_1^2}{\l_0\b^{3/2}} \sqrt{\hat{E}^m}  \left(  \b D^m_h + \b D^m_\bu) \right) +  \frac{4AC_1^2}{\l_0\b}\sqrt{\hat{E}^m}  \(\e D^m_\s\)+ \frac{4AC_1^2}{\l_0}\sqrt{\hat{E}^m}  D^m_\p.
\end{align}
So if the initial data satisfies the condition
\begin{align}
& \sqrt{\hat{E}^m(0)} \leq \min \left\{ \frac{\frac{\l_0}{4\e^2}}{\frac{20AC_1^2}{\l_0\b^{3/2}}}, \frac{\frac{\l_0}{4\e}}{\frac{4AC_1^2}{\l_0\b}}, \frac{\frac{\l_0}{2}}{\frac{4AC_1^2}{\l_0}} \right\}\nonumber\\
& \hat{E}^m(0) \leq \(\frac{\l_0^2}{80AC_1^2}\)^2 ,\nonumber\\
\text{or equivalently, }\quad &E^m_h(0) +\frac{1}{\b}E^m_\p(0)\leq \frac{2\l_0^3}{(80AC_1)^2\e} ,
\end{align}
then similar to (\ref{ineq of estimates}) - (\ref{ineq of estimates_2}), one has,
\begin{align}
&\frac{1}{2}\hat{E}^m + \frac{\l_0}{4}\(E^m_h + E^m_\p\) \leq 0\nonumber\\
&\e E^m_h(t) + E^m_\p(t) \leq -\int_0^t E^m_h(s) + E^m_\p(s) ds + \frac{3}{\l_0}\(\e E^m_h(0) + E^m_\p(0)\)\nonumber\\
&E^m_h(t) \leq \frac{3}{\l_0}e^{-\frac{t}{\e}}\( E^m_h(0) + \frac{1}{\e}E^m_\p(0)\), \qd E^m_\p(t) \leq   \frac{3}{\l_0}e^{-t}\(\e E^m_h(0) + E^m_\p(0)\).
\end{align}
This completes the proof of (\ref{conclusion: PR}) in Theorem \ref{expo decay}.

\setcounter{section}{-1}

\appendix
\section*{Appendices}
\addcontentsline{toc}{section}{Appendices}
\renewcommand{\thesubsection}{\Alph{subsection}}
\renewcommand{\thesection}{A}

\section{The proof of Lemma \ref{ineqs}}
\numberwithin{equation}{section}
\label{proof of ineqs}

\counterwithin{theorem}{section}

\begin{proof}
\begin{itemize}
\item [(a)] By the definition of $\bu$ in (\ref{def u}), and (\ref{with e on t_2}), (\ref{eqn for u with e_1}),
\begin{align}
 &\la \pt^{\bk}\pt_x\p, v\sM\pt^{\bk}h \ra =  \la \pt^{\bk}\pt_x\p,\pt^{\bk}\bu \ra =   -\la \pt^{\bk}\p,\pt^{\bk}\pt_x\bu \ra  =\d\la  \pt^{\bk}\p, \pt^{\bk}\pt_t\s \ra \nonumber\\
  =&  -\d\la  \pt^{\bk}\p, \pt^{\bk}\pt_x^2\pt_t\p \ra = \d\la  \pt^{\bk}\pt_x\p, \pt^{\bk}\pt_x\pt_t\p \ra  = \frac{\d}{2}\pt_t\ll  \pt^{\bk}\pt_x\p \rl^2,
\end{align}
where the last equality of the first line is because of (\ref{eqn for u with e_1}), and the first equality of the second line is because of (\ref{with e on t_2}).
\item [(b)] First break $\pt^{\bk}h = \pt^{\bk}\s\sM + (\pt^{\bk}h - \pt^{\bk}\s\sM)$, and then use $\pt^{\bk}h - \pt^{\bk}\s\sM = \pt^{\bk}\bu\, v\sM + (1-\Pi)\pt^{\bk}h$, one has,
\begin{align}
&\la  \pt^{\bk} \pt_x\p\, \pt_v( \pt^{\bi}h),  \pt^{\bl}h \ra \nonumber\\
=& \la \pt^{\bk} \pt_x\p\, \pt_v( \pt^{\bi} \s\sM), \pt^{\bl} h  \ra + \la \pt^{\bk} \pt_x\p\,  \pt_v(\pt^{\bi} h - \pt^{\bi} \s \sM), \pt^{\bl} \s\sM + \left(\pt^{\bl}h -  \pt^{\bl} \s\sM\right)\ra \nonumber\\
= &-\frac{1}{2}\la \pt^{\bk} \pt_x\p\, \pt^{\bi} \s, \pt^{\bl}h v\sM \ra +  \la \pt^{\bk} \pt_x\p\,  \pt_v(\pt^{\bi} h) , \pt^{\bl} \s\sM\ra -  \la \pt^{\bk} \pt_x\p\,  \pt_v( \pt^{\bi} \s \sM), \pt^{\bl} \s\sM\ra \nonumber\\
 &+ \la \pt^{\bk} \pt_x\p, \pt_v\left(\pt^{\bi} \bu v\sM+ (1-\Pi)\pt^{\bi} h\right)\left(\pt^{\bl} \bu v\sM+ (1-\Pi)\pt^{\bl} h\right) \ra\nonumber\\
 \leq &-\frac{1}{2}\la \pt^{\bk}  \pt_x\p, \pt^{\bi} \s \pt^{\bl} \bu\ra -  \la \pt^{\bk}  \pt_x\p\, \pt^{\bi} h, \pt^{\bl} \s \pt_v(\sM)\ra + \frac{1}{2}\la \pt^{\bk} \pt_x\p,   \pt^{\bi} \s \pt^{\bl} \s vM\ra\nonumber\\
 &+\ll \pt^{\bk}  \pt_x\p\rl_{L^\infty_{x,z}}\( \ll \pt^{\bi} \bu \pt_v(v\sM)\rl^2 +  \ll \pt_v(1-\Pi)\pt^{\bi} h\rl^2 +  \ll \pt^{\bl} \bu(v\sM) \rl^2 + \ll (1-\Pi)\pt^{\bl} h\rl^2\)\nonumber\\
 \leq&-\frac{1}{2}\la \pt^{\bk}  \pt_x\p, \pt^{\bi} \s \pt^{\bl} \bu\ra +\frac{1}{2}\la \pt^{\bk}  \pt_x\p, \pt^{\bl} \s \pt^{\bi} \bu\ra + 0 \nonumber\\
&+C_1\ll \pt^{\bk}  \pt_x\p\rl_{H^1_z(H^1_x)}\( \frac{3}{4}\ll \pt^{\bi} \bu \rl^2 +  \ll (1-\Pi)\pt^{\bi} h\rl^2_\nu +  \ll \pt^{\bl}\bu \rl^2 + \ll (1-\Pi)\pt^{\bl} h\rl^2_\nu\),
\label{pc_1}
\end{align}
where the last inequality comes from the Sobolev embedding for 1D,
\begin{align}
\ll f\rl_{C^0_x} \leq C_1^2\ll f \rl_{H^1_x}, \qd \ll f\rl_{C^0_x} \leq C_1\ll f \rl_{H^1_z},\qd\ll f \rl_{C^0_{x,z}} \leq C_1\ll f \rl_{H^1_z(H^1_x)},  \quad \forall f\in H^1_z(H^1_x),
\label{def:Ck}
\end{align}
for some constant $C_1\geq 1$. \\
Next,
\begin{align}
&- \frac{1}{2} \la v\pt^{\bk} \pt_x\p\, \pt^{\bi} h,  \pt^{\bl}h \ra\nonumber\\
=&-\frac{1}{2}\la v\pt^{\bk}\pt_x\p\,\pt^{\bi}\s\sM, \pt^{\bl}h \ra -\frac{1}{2}\la v\pt^{\bk}\pt_x\p\left(\pt^{\bi}h - \pt^{\bi}\s\sM\right) , \pt^{\bl} \s\sM + \left(  \pt^{\bl} h - \pt^{\bl} \s\sM\right) \ra\nonumber\\
=&-\frac{1}{2}\la \pt^{\bk}\pt_x\p, \pt^{\bi}\s \pt^{\bl}\bu \ra  - \frac{1}{2}\la \pt^{\bk}\pt_x\p, v\pt^{\bi}h\, \pt^{\bl}\s\sM \ra + \frac{1}{2}\la \pt^{\bk}\pt_x\p, v\pt^{\bi}\s\sM\, \pt^{\bl}\s\sM\ra\nonumber\\
&-\frac{1}{2}\la \pt^{\bk}\pt_x\p, v\left( \pt^{\bi}\bu v\sM + (1-\Pi)\pt^{\bi}h  \right) \left( \pt^{\bl}\bu v\sM + (1-\Pi)\pt^{\bl}h \right) \ra\nonumber\\
\leq & -\frac{1}{2}\la \pt^{\bk}\pt_x\p, \pt^{\bi}\s \pt^{\bl}\bu \ra  - \frac{1}{2}\la \pt^{\bk}\pt_x\p, \pt^{\bi}\bu \pt^{\bl}\s \ra + 0\nonumber\\
&+\frac{1}{2} \ll \pt^{\bk}\pt_x\p \rl_{L^\infty_{x,z}}\left(\int |v|(\pt^{\bi}\bu v\sM)^2 d\mu  +\int |v| \left((1-\Pi)\pt^{\bi}h\right)^2d\mu   \right. \nonumber\\
&\left. + \int |v|(\pt^{\bl}\bu v\sM)^2d\mu + \int |v|((1-\Pi) \pt^{\bl}h )^2d\mu \right)\nonumber\\
\leq & -\frac{1}{2}\la \pt^{\bk}\pt_x\p, \pt^{\bi}\s \pt^{\bl}\bu \ra  - \frac{1}{2}\la \pt^{\bk}\pt_x\p, \pt^{\bi}\bu \pt^{\bl}\s \ra\nonumber\\
&+\frac{1}{2}C_1\ll \pt^{\bk}\pt_x\p \rl_{H^1_z(H^1_x)} \left(2\ll \pt^{\bi} \bu\rl^2+ \frac{1}{2}\ll (1-\Pi)\pt^{\bi}h \rl^2_\nu  + 2\ll \pt^{\bl}\bu \rl^2 +\frac{1}{2}\ll (1-\Pi) \pt^{\bl}h \rl^2_\nu \right).\nonumber\\
\label{pc_2}
\end{align}
Therefore (\ref{pc_1}) + (\ref{pc_2}) gives,
\begin{align}
 &\la \pt^{\bk} \pt_x\p\, \pt_v(\pt^{\bi}h), \pt^{\bl}h \ra - \frac{1}{2} \la v\pt^{\bk} \pt_x\p\, \pt^{\bi} h,  \pt^{\bl}h \ra\nonumber\\
 \leq& -\la \pt^{\bk}\pt_x\p, \pt^{\bi}\s \pt^{\bl}\bu \ra +C_1\ll \pt^{\bk}\pt_x\p \rl_{H^1_z(H^1_x)} \left(2\ll \pt^{\bi} \bu\rl^2+ 2\ll (1-\Pi)\pt^{\bi}h \rl^2_\nu + 2\ll \pt^{\bl}\bu \rl^2  +2\ll (1-\Pi) \pt^{\bl}h \rl^2_\nu \right)\nonumber\\
 \leq& C_1\ll \pt^{\bk}\pt_x\p \rl_{H^1_z(H^1_x)} \left(a\ll\pt^{\bi}\s \rl^2 +  2\ll \pt^{\bi} \bu\rl^2 + 2\ll (1-\Pi)\pt^{\bi}h \rl^2_\nu + \(2+\frac{1}{a}\)\ll \pt^{\bl}\bu \rl^2 +2\ll (1-\Pi) \pt^{\bl}h \rl^2_\nu \right)\nonumber,
 \end{align}
 for $a$ to be determined later.

{\item [(c)] 
For the term $\la \pt^{\bk} \pt_x\p, \pt_v\left(\pt^{\bi} \bu v\sM+ (1-\Pi)\pt^{\bi} h\right)\left(\pt^{\bl} \bu v\sM+ (1-\Pi)\pt^{\bl} h\right) \ra$, one can also bounded by,
\begin{align}
	&\la \pt^{\bk} \pt_x\p, \pt_v\left(\pt^{\bi} \bu v\sM+ (1-\Pi)\pt^{\bi} h\right)\left(\pt^{\bl} \bu v\sM+ (1-\Pi)\pt^{\bl} h\right) \ra\nonumber\\
	\leq &  \int_{I_z} \ll \pt^{\bk}\pt_x\p \rl_{L^\infty_x} \ll \pt_v\(\pt^{\bi} \bu v\sM+ (1-\Pi)\pt^{\bi} h\) \rl_{L^2_{x,v}} \ll \pt^{\bl} \bu v\sM+ (1-\Pi)\pt^{\bl} h \rl_{L^2_{x,v}}  d\mu(z)\nonumber\\
	\leq & C_1 \ll \ll \pt_v\(\pt^{\bi} \bu v\sM+ (1-\Pi)\pt^{\bi} h \)\rl_{L^2_{x,v}}\rl_{L^\infty_z} \sqrt{\ll \pt^{\bk}\pt_x\p \rl^2 + \ll \pt^{\bk}\pt_x^2\p \rl^2} \ll  \pt^{\bl} \bu v\sM+ (1-\Pi)\pt^{\bl} h \rl\nonumber\\
	\leq & C_1^2  \sqrt{\ll \pt^{\bk}\pt_x\p \rl^2 + \ll \pt^{\bk}\pt_x^2\p \rl^2} \( \frac{1}{2}\sum_{i\leq 1}\ll\pt_v\( \pt^{\bi} \bu v\sM+ (1-\Pi)\pt^{\bi}\pt_z^i h\) \rl^2 + \frac{1}{2}\ll  \pt^{\bl} \bu v\sM+ (1-\Pi)\pt^{\bl} h\rl^2\)\nonumber\\
	\leq & C_1^2  \sqrt{\ll \pt^{\bk}\pt_x\p \rl^2 + \ll \pt^{\bk}\pt_x^2\p \rl^2}\( \frac{3}{4}\sum_{i\leq 1}\ll \pt^{\bi}\pt_z^i \bu \rl^2 + \sum_{i\leq 1}\ll (1-\Pi)\pt^{\bi}\pt_z^i h \rl_\nu^2 + \ll  \pt^{\bl} \bu \rl^2 + \ll (1-\Pi)\pt^{\bl} h\rl_\nu^2\)
\end{align}
Similarly,
\begin{align}
    &-\frac{1}{2}\la \pt^{\bk}\pt_x\p, v\left( \pt^{\bi}\bu v\sM + (1-\Pi)\pt^{\bi}h  \right) \left( \pt^{\bl}\bu v\sM + (1-\Pi)\pt^{\bl}h \right) \ra\nonumber\\
    \leq & \frac{C_1^2}{2} \sqrt{\ll \pt^{\bk}\pt_x\p \rl^2 +\ll \pt^{\bk}\pt^2_x\p \rl^2} \( 2\sum_{i\leq 1}\ll \pt^{\bi}\pt_z^i \bu\rl^2 + \frac{1}{2}\sum_{i\leq 1}\ll (1-\Pi)\pt^{\bi}\pt_z^ih \rl^2_\nu  \right.\nonumber\\
    &\left. + 2\ll \pt^{\bl}\bu \rl^2 +\frac{1}{2}\ll (1-\Pi) \pt^{\bl}h \rl^2_\nu \)
\end{align}
\begin{align}
    &-\la \pt^{\bk}\pt_x\p, \pt^{\bi}\s \pt^{\bl}\bu \ra \leq  C_1^2  \sqrt{\ll \pt^{\bk}\pt_x\p \rl^2 + \ll \pt^{\bk}\pt_x^2\p \rl^2} \( a\sum_{i\leq 1}\ll\pt^{\bi}\pt_z^i \s \rl^2 + \frac{1}{a}\ll  \pt^{\bl} \bu \rl^2\)
\end{align}
which gives,
\begin{align}
 &\la \pt^{\bk} \pt_x\p\, \pt_v(\pt^{\bi}h), \pt^{\bl}h \ra - \frac{1}{2} \la v\pt^{\bk} \pt_x\p\, \pt^{\bi} h,  \pt^{\bl}h \ra\nonumber\\
 \leq&  C_1^2  \sqrt{\ll \pt^{\bk}\pt_x\p \rl^2 + \ll \pt^{\bk}\pt_x^2\p \rl^2} \left(a\sum_{i\leq 1}\ll\pt^{\bi}\pt_z^i \s \rl^2 +  2\sum_{i\leq 1}\ll \pt^{\bi}\pt_z^i \bu\rl^2 + 2\sum_{i\leq 1}\ll (1-\Pi)\pt^{\bi}\pt_z^ih \rl^2_\nu \right.\nonumber\\
 &\left. + \(2+\frac{1}{a}\)\ll \pt^{\bl}\bu \rl^2 +2\ll (1-\Pi) \pt^{\bl}h \rl^2_\nu \right)\nonumber,
 \end{align}
 for $a$ to be determined later.
}

 \item [(d)] Since
 \begin{align}
&\la  \pt^{\bk} \pt_x\p\, \pt_v( \pt^{\bi}h),  \pt^{\bl}h \ra \nonumber\\
=& \la \pt^{\bk} \pt_x\p\, \pt_v( \pt^{\bi} h), \pt^{\bl} \s\sM  \ra + \la \pt^{\bk} \pt_x\p\,  \pt_v(\pt^{\bi} h), \pt^{\bl} h - \pt^{\bl}  \s \sM\ra \nonumber\\
= &-\la \pt^{\bk} \pt_x\p\, \pt^{\bi} h, \pt^{\bl} \s\pt_v(\sM)  \ra - \la \pt^{\bk} \pt_x\p\,  \pt^{\bi} h, \pt_v(  \pt^{\bl}\bu v\sM +  (1-\Pi)\pt^{\bl}h)\ra \nonumber\\
\leq &\frac{1}{2}\la \pt^{\bk} \pt_x\p, \pt^{\bi} \bu\pt^{\bl} \s  \ra + \la \lv \pt^{\bi} h\rv, \, \lv \pt^{\bk}\pt_x\p\(\pt^{\bl}\bu \pt_v( v\sM) + \pt_v(1-\Pi)\pt^{\bl}h\)\rv \ra  \nonumber\\
\leq &\frac{1}{2}\la \pt^{\bk}\pt_x\p, \pt^{\bl}\s \pt^{\bi}\bu \ra + \ll\lv  \pt^{\bi} h\rv\rl_{L^\infty_{x,z}} ( \frac{a}{2}\ll\pt^{\bk}\pt_x\p\rl^2 +  \frac{1}{a}\ll\pt^{\bl} \bu \pt_v(v\sM) \rl^2 +\frac{1}{a}\ll \pt_v(1-\Pi)\pt^{\bl}h\rl^2)  \nonumber\\
\leq &\frac{1}{2}\la \pt^{\bk}\pt_x\p, \pt^{\bl}\s \pt^{\bi}\bu \ra + C_1\ll \pt^{\bi} h\rl_{H^1_z(H^1_x)} (\frac{a}{2} \ll\pt^{\bk}\pt_x\p\rl^2 +  \frac{3}{4a}\ll\pt^{\bl} \bu  \rl^2 + \frac{1}{a}\ll (1-\Pi)\pt^{\bl}h\rl^2_\nu),
\label{pc_3}
\end{align}

Next, similar to (\ref{pc_3}),
 \begin{align}
 &- \frac{1}{2} \la v\pt^{\bk} \pt_x\p\, \pt^{\bi} h,  \pt^{\bl}h \ra = - \frac{1}{2} \la \pt^{\bk} \pt_x\p, v\pt^{\bi} h\,  \pt^{\bl}h \ra\nonumber\\
 =&-\frac{1}{2}\la \pt^{\bk}\pt_x\p,v \pt^{\bi}h\, \pt^{\bl}\s \sM \ra -\frac{1}{2}\la \pt^{\bk}\pt_x\p, v\pt^{\bi}h  \left(\pt^{\bl} \bu v\sM +  (1-\Pi) \pt^{\bl} h\right)\ra\nonumber\\
 \leq &-\frac{1}{2}\la \pt^{\bk}\pt_x\p, \pt^{\bl}\s \pt^{\bi}\bu \ra +\frac{1}{2}\la  \lv \pt^{\bi} h\rv,\,  \lv\pt^{\bk}\pt_x\p\(\pt^{\bl}\bu v^2\sM + v(1-\Pi)\pt^{\bl}h\)\rv \ra  \nonumber\\
\leq &-\frac{1}{2}\la \pt^{\bk}\pt_x\p, \pt^{\bl}\s \pt^{\bi}\bu \ra + \frac{C_1}{2}\ll \pt^{\bi} h\rl_{H^1_z(H^1_x)} ( \frac{a}{2}\ll\pt^{\bk}\pt_x\p\rl^2 + \frac{1}{a} \ll\pt^{\bl} \bu (v^2\sM) \rl^2 +\frac{1}{a}\ll v(1-\Pi)\pt^{\bl}h\rl^2)  \nonumber\\
\leq &-\frac{1}{2}\la \pt^{\bk}\pt_x\p, \pt^{\bl}\s \pt^{\bi}\bu \ra + \frac{C_1}{2}\ll \pt^{\bi} h\rl_{H^1_z(H^1_x)} (\frac{a}{2} \ll\pt^{\bk}\pt_x\p\rl^2 +  \frac{3}{a}\ll\pt^{\bl} \bu  \rl^2 + \frac{1}{a}\ll (1-\Pi)\pt^{\bl}h\rl^2_\nu).
\label{pc_4}
 \end{align}
 Therefore (\ref{pc_3}) + (\ref{pc_4}) gives,
\begin{align}
 &\la \pt^{\bk} \pt_x\p\, \pt_v(\pt^{\bi}h), \pt^{\bl}h \ra - \frac{1}{2} \la v\pt^{\bk} \pt_x\p\, \pt^{\bi} h,  \pt^{\bl}h \ra\nonumber\\
 \leq&  C_1\ll \pt^{\bi} h\rl_{H^1_z(H^1_x)} (a \ll\pt^{\bk}\pt_x\p\rl^2 +  \frac{3}{a}\ll\pt^{\bl} \bu  \rl^2 + \frac{2}{a}\ll (1-\Pi)\pt^{\bl}h\rl^2_\nu).
  \end{align}

{\item [(e)]
Similar to the proof in (c), based on the estimates in (d), one can bound the same term by,
\begin{align}
     &\la \pt^{\bk} \pt_x\p\, \pt_v(\pt^{\bi}h), \pt^{\bl}h \ra - \frac{1}{2} \la v\pt^{\bk} \pt_x\p\, \pt^{\bi} h,  \pt^{\bl}h \ra\nonumber\\
     \leq&  C_1^2\sqrt{ \ll \pt^{\bi} h\rl^2 + \ll \pt^{\bi} \pt_xh\rl^2} (a \sum_{i\leq 1}\ll\pt^{\bk}\pt_z^i\pt_x\p\rl^2 +  \frac{3}{a}\ll\pt^{\bl} \bu  \rl^2 + \frac{2}{a}\ll (1-\Pi)\pt^{\bl}h\rl^2_\nu).
\end{align}
}
 
 \item [(f)] By (\ref{eqn for u with e_1}) and (\ref{with e on t_1}) one derives,
\begin{align}
\pt_x(\pt^{\bk}\pt_x\pt_t\p) = - \pt^{\bk}\pt_t\s = \pt_x(\frac{1}{\d}\pt^{\bk}\bu),
\end{align}
integrating it from $-\infty$ to x implies,
\begin{align}
&\pt^{\bk}\pt_x\pt_t\p(x) = \frac{1}{\d}\pt^{\bk}\bu(x),
\end{align}
Hence,
\begin{align}
\ll\pt^{\bk}\pt_x\pt_t\p \rl^2\leq \frac{1}{\d^2} \ll\pt^{\bk}\bu \rl^2.
\end{align}

\end{itemize}

\end{proof}

\renewcommand{\thesection}{B}
\section{The proof of Theorem \ref{new energy}}

\label{proof of new energy}
\begin{proof}

The proof of Theorem \ref{new energy} is almost the same as the proof of Theorem \ref{expo decay}, expect for the estimates on the nonlinear term $\sum_{j = 0}^1\sum_{l = 0}^{m}\pt_x^j\pt_z^{l}\(\pt_x\p\, \pt_vh - \frac{v}{2} \pt_x\p\,h\)$, and, $\sum_{j = 0}^1\sum_{l = 0}^{m}\pt_x^j\pt_z^{l}\(\pt_x\p \s\)$. We first estimate the case of $j=0$. \\
Taking $\pt_z^l$ on (\ref{with e on t_1}), and multiplying by $\(\frac{l+1}{l!}\)^2\pt_z^l h$, then integrating it over $\mu(x,v,z)$, one has,
\begin{align}
&\frac{\d}{2}\pt_t\(\e\ll \frac{l+1}{l!}\pt_z^lh\rl^2 + \d\ll \frac{l+1}{l!}\pt_z^l\pt_x\p\rl^2\) + \l_0\ll\frac{l+1}{l!}(1-\Pi)\pt_z^lh \rl_\nu^2 + \ll\frac{l+1}{l!}\pt_z^lu \rl^2\nonumber\\
\leq & \d\sum_{i = 0}^l \frac{l!}{i!(l-i)!} \frac{l+1}{l!} \la \pt_z^{l-i}  \pt_x\p \(\pt_v - \frac{v}{2}\)\pt_z^ih, \frac{l+1}{l!}\pt_z^lh \ra\nonumber\\
= & \d\sum_{i = 0}^l\frac{ l+1}{(l-i+1)(i+1)} \la\(\frac{l-i+1}{(l-i)!} \pt_z^{l-i}\pt_x\p \) \(\frac{i+1}{i!} \pt_z^ih\), \frac{l+1}{l!}\pt_z^lh \ra\nonumber\\
\leq &\d\sum_{i = 0}^l\(\frac{ 1}{i+1}+ \frac{ 1}{l-i+1}\) \la\(\frac{l-i+1}{(l-i)!} \pt_z^{l-i}\pt_x\p \) \(\frac{i+1}{i!} \pt_z^ih\), \frac{l+1}{l!}\pt_z^lh \ra\nonumber\\
\leq &\d C_1\sum_{i = 0}^{l-1}\frac{ 1}{i+1}
\ll \frac{i+1}{i!} \pt_z^ih\rl_{H^1_z(H^1_x)} \left(  \frac{3}{a}\ll\frac{l+1}{l!} \pt_z^lu  \rl^2 + \frac{2}{a}\ll \frac{l+1}{l!} (1-\Pi)\pt_z^lh\rl^2_\nu\right.\nonumber\\
&\left.+ a\ll\frac{l-i+1}{(l-i)!} \pt_z^{l-i}\pt_x\p\rl^2 \right)\nonumber\\
& + \d C_1\sum_{i = 0}^{l-1}\frac{ 1}{l-i+1}
\ll \frac{l-i+1}{(l-i)!} \pt_z^{l-i}\pt_x\p \rl_{H^1_z(H^1_x)} \left( \(2+ \frac{1}{a}\)\ll\frac{l+1}{l!} \pt_z^lu  \rl^2 + 2\ll \frac{l+1}{l!} (1-\Pi)\pt_z^lh\rl^2_\nu\right.\nonumber\\
&\left.+ a\ll\frac{i+1}{i!} \pt_z^i\s\rl^2 + 2\ll\frac{i+1}{i!} \pt_z^iu\rl^2 + 2\ll\frac{i+1}{i!} (1-\Pi)\pt_z^ih\rl^2 \right)\nonumber\\
&+\d C_1^2\sqrt{ \ll \frac{l+1}{l!} \pt_z^lh\rl^2 + \ll\frac{l+1}{l!} \pt_z^l\pt_xh\rl^2} \(a \sum_{j\leq 1}\ll \pt_x^j\pt_x\p\rl^2\right.\nonumber\\
&\left. +  \frac{3}{a}\ll\frac{l+1}{l!} \pt_z^lu  \rl^2 + \frac{2}{a}\ll\frac{l+1}{l!} (1-\Pi)\pt_z^lh\rl^2_\nu\), \label{ineq_1}
\end{align}
where the last inequality comes from Lemma \ref{ineqs} (d), (b), (e). Notice
\begin{align}
	\(\sum_{i = 0}^{l-1}\frac{a_i}{i+1}\)^2 =& \sum_{i,j = 0}^{l-1}\frac{a_ia_j}{(i+1)(j+1)} \leq\frac{1}{2} \sum_{i ,j = 0}^{l-1} \(\frac{(i+1)^2}{(j+1)^2}\frac{a_i^2}{(i+1)^2} + \frac{(j+1)^2}{(i+1)^2}\frac{a_j^2}{(j+1)^2}\) \nonumber\\
	=& \sum_{i = 0}^{l-1}\(\sum_{j =0}^{l-1} \frac{1}{(j+1)^2}\)a_i^2 \leq (A')^2 \sum_{i = 0}^{l-1}a_i^2, 
\end{align}
where $A' = \sqrt{\sum_{i = 0}^\infty\frac{1}{(i+1)^{2}}}$. Then one can bound, 
\begin{align}
	& \sum_{i = 0}^{l-1} \frac{1}{i+1}\ll \frac{i+1}{i!} \pt_z^ih\rl_{H^1_z(H^1_x)} \leq 2A'\sqrt{\t{E}^l_h} \leq 2A' \sqrt{\t{E}^m_h},
\end{align}
and,
\begin{align}
	&\sum_{l = 1}^m\sum_{i = 0}^{l-1}\frac{ 1}{i+1}\ll \frac{i+1}{i!} \pt_z^ih\rl_{H^1_z(H^1_x)}\ll\frac{l-i+1}{(l-i)!} \pt_z^{l-i}\pt_x\p\rl^2\nonumber\\
	\leq & \(\sum_{i = 0}^{m-1}\frac{1}{i+1}\ll \frac{i+1}{i!} \pt_z^ih\rl_{H^1_z(H^1_x)}\)\(\sum_{i = 0}^m\ll\frac{i+1}{i!} \pt_z^i\pt_x\p\rl^2 \) \leq 2A'\sqrt{\t{E}^m_h} \t{D}^m_\p.
\end{align}
Here $\t{D}^m_\p$ is the corresponding new dissipation of $\pt_x\p$ in the new norm, similar for $\t{D}^m_\s, \t{D}^m_u$. All other terms in (\ref{ineq_1}) can be similarly bounded. Thus $\sum_{l = 0}^m(\ref{ineq_1})$ gives, 
\begin{align}
	&\frac{\d}{2}\pt_{t} \left[ \b \t{E}^{m,0}_h + \d \t{E}^{m,0}_\p\right] + \l_0\t{D}^{m,0}_h + \t{D}^{m,0}_u  \nonumber\\
\leq& 2\d A'C_1\sqrt{\t{E}^m_h}\(\frac{3}{a}\t{D}^m_u +\frac{2}{a} \t{D}^m_h + a\t{D}^m_\p\)  + 2\d A'C_1\sqrt{\t{E}^m_\p}\( \(4+ \frac{1}{a}\)\t{D}^m_u+ 4\t{D}^m_h+ a\t{D}^m_\s\)\nonumber\\
	&+ 2\d A'C_1^2\sqrt{\t{E}^m_h} \(a\t{D}^m_\p + \frac{3}{a}\t{D}^m_u + \frac{2}{a}\t{D}^m_h\)\nonumber\\
	\leq & 4\d A'C_1\sqrt{\t{E}^m_h}\(\frac{3}{a}\t{D}^{m,0}_u +\frac{2}{a} \t{D}^{m,0}_h + a\t{D}^m_\p\) + 2\d A'C_1\sqrt{\t{E}^m_\p}\( \(4+ \frac{1}{a}\)\t{D}^{m,0}_u+ 4\t{D}^{m,0}_h+ a\t{D}^{m,0}_\s\).\nonumber\\
\end{align}
One can use similar method to bound other nonlinear terms. We omit the details here. Now one can see that the constant $A'$ is independent of $m$, which leads to the independence of $m$ in the initial condition.

\end{proof}

\bibliographystyle{plain}

\end{document}